\documentclass[reqno,10pt,final]{amsart}
\usepackage[dvipsnames,usenames]{color} 
\usepackage[toc,page]{appendix}

\usepackage{graphicx}  
\usepackage{mathcomp,wasysym}  

\usepackage{hyperref}
\usepackage[mathscr]{euscript}  
         
\usepackage{cite}      
\usepackage{mathtools} 
\usepackage{graphicx}  
\usepackage[all]{xy} \xyoption{arc} \xyoption{color}

\usepackage{amsmath} 
\usepackage{amsthm} 
\usepackage{amsfonts}  
\usepackage{amssymb}
\usepackage{bm} 

\usepackage{verbatim}  
\usepackage{soul}   
\usepackage{xcolor}   
\usepackage{enumerate}

\usepackage{mathrsfs}

\usepackage{showkeys}

\def\Xint#1{\mathchoice
 {\XXint\displaystyle\textstyle{#1}}%
 {\XXint\textstyle\scriptstyle{#1}}%
 {\XXint\scriptstyle\scriptscriptstyle{#1}}%
 {\XXint\scriptscriptstyle\scriptscriptstyle{#1}}%
 \!\int}
 \def\XXint#1#2#3{{\setbox0=\hbox{$#1{#2#3}{\int}$}
 \vcenter{\hbox{$#2#3$}}\kern-.5\wd0}}
 
 \def\dashint{\Xint-}

\newcommand{\RR}{\mathbb R}
\newcommand{\Rd}{{\mathbb R}^d}

\newcommand{\NN}{\mathbb N}

\newcommand{\Zd}{{\mathbb Z^d}}

\newcommand{\Td}{{\mathbb T^d}}

\newcommand{\N}{\mathbb{N}}
\newcommand{\T}{\mathbb T}
\newcommand{\Z}{\mathbb Z}
\newcommand{\pat}{\partial_t}

\newcommand{\divv}{{\rm div \,}}
\newcommand{\A}{\mathcal{A} }
\newcommand{\D}{Dv }
\newcommand{\Du}{Du }
\newcommand{\cN}{{\mathcal{N}} }
\newcommand{\cS}{{\mathcal{S}} }
\newcommand{\Do}{Du_1 }

\newcommand{\vp}{u_{p}}
\newcommand{\vc}{u_{c}}
\newcommand{\divn}{\mathrm{div}^{-1}}
\newcommand{\divnd}{\mathrm{div}^{-2}}
\newcommand{\idivn}{{\mathcal{R}}}

\newcommand{\Id}{{{\rm Id}}}

\newcommand{\supp}{{{\rm supp}\,}}

\newcommand{\curl}{{\rm curl \,}}
\newcommand{\tr}{{\rm tr \,}}

\newcommand{\vertiii}[1]{{\left\vert\kern-0.25ex\left\vert\kern-0.25ex\left\vert #1 
    \right\vert\kern-0.25ex\right\vert\kern-0.25ex\right\vert}}

\newcommand{\cR}{\mathring {R}}

\newcommand{\Po}{P_{\neq 0}}

\newcounter{comentcount}
\newcounter{teocount}
\newtheorem{lem}{Lemma}
\newtheorem{prop}{Proposition}

\newtheorem{defi}{Definition}

\newtheorem{remark}{Remark}
\newtheorem{theoremmain}{Theorem}

\title[]{Non Uniqueness of power-law flows}

\author[J. Burczak]{Jan Burczak}
\address{J. Burczak: Institut f\"ur Mathematik, Universit\"at Leipzig, D-04103 Leipzig, Germany}
\author[S. Modena]{Stefano Modena}
\address{S. Modena: Technische Universit\"at Darmstadt, Fachbereich Mathematik, D-64285 Darmstadt, Germany}
\author[L. Sz\'ekelyhidi]{L\'aszl\'o Sz\'ekelyhidi}
\address{L. Sz\'ekelyhidi: Institut f\"ur Mathematik, Universit\"at Leipzig, D-04103 Leipzig, Germany}

\begin{document}

\begin{abstract}
We apply the technique of convex integration to obtain non-uniqueness and existence results for power-law fluids, in dimension $d\ge 3$. For the power index $q$ below the compactness threshold, i.e.\ $q \in (1, \frac{2d}{d+2})$, we show ill-posedness of Leray-Hopf solutions. For a wider class of indices  $q \in (1, \frac{3d+2}{d+2})$ we show ill-posedness of distributional (non-Leray-Hopf) solutions, extending the seminal paper of Buckmaster \& Vicol \cite{BV19}. In this wider class we also construct non-unique solutions for every datum in $L^2$.
\end{abstract}

\thanks{The authors are thankful to Tobias Barker, Giacomo Canevari, Eduard Feireisl, Francisco Gancedo, Martina Hofmanova, Josef M\'alek, and Angkana R\"uland for their interest in this work. The authors thank Helen Wilson for sharing her applied expertise. \\
J. B. was supported by the National Science Centre, Poland (NCN) grant SONATA 2016/21/D/ST1/03085. \\
S.M. and L.~Sz.~were partially supported by the European Research Council (ERC) under the European Union's Horizon 2020 research and innovation programme (grant agreement No.724298-DIFFINCL). \\
This work was initiated at the Hausdorff Research Institute (HIM) in Bonn during the Trimester Programme Evolution of Interfaces. The authors gratefully acknowledge the warm hospitality of HIM during this time.}
\maketitle 


\section{Introduction}
This paper studies non-uniqueness and existence of  solutions of the following model of non-Newtonian flows in $d$ dimensions, $d \geq 3$
\begin{equation}\label{eq:pnse}
\begin{aligned}
\pat v + \divv (v\otimes v)  - \divv \A(\D)+\nabla \tilde\pi &= 0, \\
\divv v &= 0, &  \\
v_{t =0} &= v_0,
\end{aligned}
\end{equation}
where the velocity field $v$ and and the pressure $\tilde \pi$ are the unknowns, $\D = \frac12 (\nabla v + \nabla^T v)$,
and the non-Newtonian tensor $\A$ is given by the following power law
\begin{equation}\label{eq:orl}
\A (Q) = (\nu_0+\nu_1 |Q|)^{q-2} Q, 
\end{equation}
for some $\nu_0, \nu_1 \geq 0$ and $q \in (1, \infty)$. A natural energy associated with the system \eqref{eq:pnse} is
\begin{equation}
\label{eq:energy}
e(t) = \int  |v(t)|^2 + 2 \int_0^t \int  \A \big(D v (s)\big) D v (s) ds.
\end{equation}
Let us consider a distributional solution $v$ to \eqref{eq:pnse},  \eqref{eq:orl} with spatial mean zero, on a $d$-dimensional flat torus. The formula \eqref{eq:energy} together with $\A(Q) Q \sim |Q|^q$ explains why $v \in L^\infty (L^2) \cap L^q (W^{1,q})$ is called an \emph{energy solution}. If such solution satisfies additionally the energy inequality $e(t) \leq e(0)$ ($t$-a.e.), then it is called  a \emph{Leray-Hopf solution}.
 
For the problem \eqref{eq:pnse} we show two non-uniqueness and one existence result. In short:
\begin{enumerate}[(A)]
\item \label{bcom} In the regime $1< q<2d/(d+2)$: There are non-unique Leray-Hopf solutions.
\item \label{bsc} In the regime $1< q < (3d+2)/(d+2)$: There are non-unique distributional solutions dissipating the kinetic part of the energy.
\item In the regime $1< q < (3d+2)/(d+2)$: For any initial datum $a \in L^2$ there are infinitely many distributional solutions of the  Cauchy problem.
\end{enumerate}

Our results are sharp concerning the power-law index $q$. The regime $1< q < (3d+2)/(d+2)$ includes the case of the incompressible Navier-Stokes equation in $d\ge3$.
The precise formulations can be found in Section \ref{ssec:oc}.

\subsection{Background of power-law flows}

Model \eqref{eq:pnse} with a slightly different choice of $\mathcal{A}(Q)$, namely
\begin{equation}
\label{eq:orl-lady}
\mathcal{A}(Q) = (\nu_0+\nu_1 |Q|^{q-2}) Q,
\end{equation}
with $q\ge 2$ was introduced to wide mathematical community by Ladyzhenskaya at her 1966 Moscow ICM speech; her formula (30) in \cite{LadICM66} corresponds exactly to \eqref{eq:pnse}, \eqref{eq:orl-lady}. With $q=2$, both models \eqref{eq:pnse},  \eqref{eq:orl} and \eqref{eq:pnse}, \eqref{eq:orl-lady} reduce to the (incompressible) Navier-Stokes equations.

The Ladyzhenskaya's choice: \eqref{eq:orl-lady} with $q \ge 2$ and our \eqref{eq:orl} with $q \ge 2$ are analytically equivalent. 
In particular, the non-Newtonian tensor $\mathcal{A}(Q)$ is  in both cases nonsingular at $Q=0$, and distributional solutions are well-defined for velocity fields in the class 
\begin{equation}
\label{eq:class-v}
v \in L^2_{loc}, \quad Dv \in L^q_{loc}.
\end{equation}
The difference between \eqref{eq:orl-lady} and \eqref{eq:orl}  plays a role for $q < 2$. Firstly, $\nu_0+\nu_1 |Q|^{q-2}$ of \eqref{eq:orl-lady} is singular at $|Q|=0$, while our $(\nu_0+\nu_1 |Q|)^{q-2}$ for $\nu_0>0$ is not. More importantly, in \eqref{eq:orl-lady} a \emph{linear} dissipation is present. Thus, distributional solutions to \eqref{eq:pnse}-\eqref{eq:orl-lady} make sense provided $Dv \in L^2_{loc}$. 
So the choice \eqref{eq:orl} isolates the `pure $L^q$-dissipation' behaviour, while \eqref{eq:orl-lady} involves `$L^2$-$L^q$ dissipation'.

Ladyzhenskaya's rationale for analysing \eqref{eq:pnse} was twofold: on the one hand, relaxation $q \ge 2$ helps to avoid the traps of the Navier-Stokes case $q =2$. At the same time, the choice of power-laws for the tensor $\A$ is both consistent with first principles of continuum mechanics and widely used in applications. Let us elaborate on each of these points.

The model \eqref{eq:pnse} with power-law for $\A$ of type \eqref{eq:orl} or  \eqref{eq:orl-lady} agrees with the constitutive relations for incompressible, viscous fluids. Recall that in deriving the Navier-Stokes equation one restricts the admissible relations between the Cauchy stress tensor $\mathcal{T}$ and $D$ (dictated by the material frame indifference) 
by the Ansatz of linear dependence between $\mathcal{T}$ and $D$ (i.e.\ by the Stokes law), cf.\ \cite{GFAbook}. The power law model relaxes this Ansatz, but remains well within the frame indifference principle. 

Of course studying an arbitrary model that is merely consistent with the first principles may be applicationally void. This is not the case of  \eqref{eq:pnse} however. The power-laws have been proposed independently in 1920's by Norton \cite{Nor29} in metallurgy and by de Waele \cite{deW23} and Ostwald \cite{Ost29} in polymer chemistry. The related timeline can be found in section 1 of \cite{OssRud14}. For details, the interested reader may consult also the monographs \cite{Sch78book, MNRR, Sar16book, BirArmHas87} and the recent survey \cite{BMR19} with its references. Just in order to fix the hydrodynamical intuition, let us observe that $q<2$ in \eqref{eq:pnse} models the case when the fluid is more viscous (roughly, `solid-like') for small shears (`external forces') and less viscous (`liquid-like') for large shears e.g.\ ice pack, ketchup, emulsion paints, hair gel, whereas $q>2$ means reverse behavior e.g.\ cornstarch-water solution, silicone-based solutions.

Let us note that, despite the mathematical interest in $q \ge 2$ in context of gaining regularity compared to Navier-Stokes equations, the `shear-thinning' case $q \le 2$ appears to be more meaningful for applications, where models of type \eqref{eq:pnse} with $\nu_0>0$ appear as Bird-Carreau-Yasuda models (or called by a subset of those names). In particular, experimental fits for the threshold value $6/5$ and above can be found on p.\ 174 of \cite{BirArmHas87}. Furthermore, even parameter choices well-into our Leray-Hopf non-uniqueness regime are suggested, cf.\ p.18 of \cite{Tan00}. (In both \cite{BirArmHas87} and \cite{Tan00} $n=q-1$, $d=3$. A discrepancy between appearing there $a$ and our model is insignificant for our results.)

From the applicational perspective, our result may be seen as invalidating certain choices of parameters and data.

\subsection{Essential analytical results for power-law fluids} \label{ssec:old-results}
Consider the system \eqref{eq:pnse}, \eqref{eq:orl}.
For $q > \frac{2d}{d+2}$ the space $W^{1,q}$ of system's energy embeds compactly into $L^2_{loc}$ of the convective term $\divv (v\otimes v)$. Hence one may expect an existence proof of Leray-Hopf solutions via compactness methods. Indeed, a relevant statement can be found in \cite{DRW10}, which is itself the final step in a chain of attempts of many authors, including Frehse and Ne\v{c}as with collaborators \cite{MNR93,FMS00} to improve the lower bound on $q$. To be precise, the energy inequality $e(t) \leq e(0)$ is not stated explicitly in \cite{DRW10}; however it can be proven e.g.\ along the lines of proof of Theorem 3.3 of \cite{BMR19}.

Observe that \eqref{eq:pnse} with $\nu_0 =0$ is invariant under the scaling 
\begin{equation}\label{eq:sca}
v_\lambda := \lambda^\alpha v ( \lambda x, \lambda^{\alpha +1} t) \quad \text{ with } \; \alpha = \frac{q-1}{3-q}.
\end{equation}
Consequently, the energy of $v_\lambda$ vanishes on small scales iff $q < \frac{3d+2}{d+2}$. This suggests that the case $q \ge \frac{3d+2}{d+2}$ of \eqref{eq:pnse} is a perturbation of the problem \eqref{eq:pnse} without the convective term. Indeed, for $q \ge \frac{3d+2}{d+2}$ uniqueness in the energy class (at least for tame initial data) holds, cf. \cite{MNRR}, section 5.4.1; see also \cite{BKP19}.

What is known about existence and uniqueness of solutions to \eqref{eq:pnse} can be thus sketched as follows 
\begin{figure}[h!]
\includegraphics[width=0.4\linewidth]{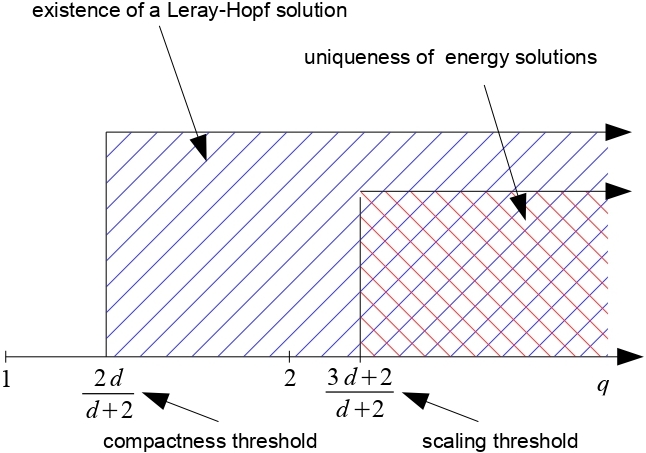}
  \caption{Known results}
  \label{fig:1}
\end{figure}
%
\subsection{Our contribution}\label{ssec:oc}
The short version of our results presented at the very beginning of the paper, recast graphically to facilitate comparison with Figure \ref{fig:1}, reads
\begin{figure}[h!]
\includegraphics[width=0.52\linewidth]{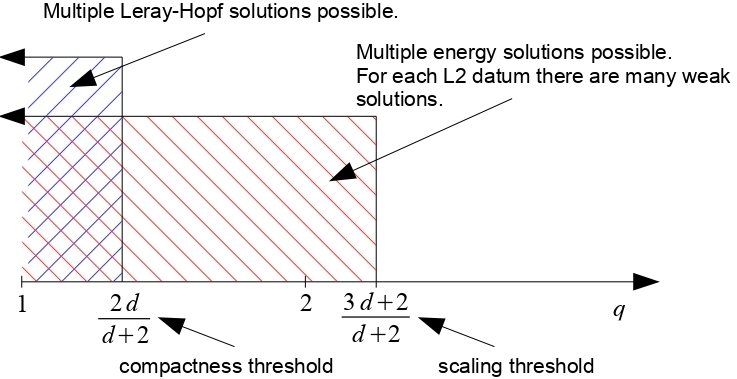}
  \caption{Our results}
  \label{fig:2}
\end{figure}

Observe that Figure \ref{fig:2} complements Figure \ref{fig:1} sharply with respect to $q$.

Let us now present the detailed statements of our results. We always consider system \eqref{eq:pnse} on the $d$-dimensional flat torus $\T^d$, with $v$ having its spatial mean zero.
\subsubsection{Non-uniqueness in the Leray-Hopf class} Our first theorem and its corollary show that below the compactness exponent, i.e.\ for $q<\frac{2d}{d+2}$, multiple Leray-Hopf solutions may emanate from the same $L^2$ initial data. In fact, we produce solutions $v \in C(L^2) \cap C (W^{1,q})$ with quite arbitrary pre-determined profile $e$ of the (total) energy \eqref{eq:energy}. 
\begin{theoremmain}\label{thm:main}
Consider \eqref{eq:pnse}, \eqref{eq:orl} on the space-time domain $\Td  \times (0,1)$. Let $q< \frac{2d}{d+2}$. Fix an arbitrary $e \in C^\infty([0,1]; [1/2,1])$. There exists $v \in C([0,1]; L^2(\T^d)) \cap C ([0,1]; W^{1,q}(\T^d))$ such that
\begin{enumerate}[1)]
\item $v$ solves \eqref{eq:pnse} distributionally, i.e.
\[
\int_0^1 \int_\Td -v \cdot \pat \varphi - v\otimes v  \nabla \varphi   + \A(\D) \nabla \varphi = 0, \qquad  \forall_{t \in [0,1]} \quad \int_\Td v (t) \cdot \nabla \psi = 0
\]
for any divergence-free $\varphi \in C^1 (\Td  \times [0,1])$ vanishing at $t=0$ and $t=1$, and any  $\psi \in C^1 (\Td)$;
\item the total energy equals $e$, i.e.\
\begin{equation}\label{eq:epr}
\int_\Td  |v|^2 (t)  + 2 \int_0^t \int_\Td  \A (D v) D v = e(t).
\end{equation}
\end{enumerate}
Moreover, fix $0 \leq T_1 < T \leq 1$ and  two energy profiles $e_1, e_2$ as above, such that $e_1 (t)= e_2 (t) $ for $t \in [0, T]$. There exists $v_1, v_2$ satisfying 1), 2) and such that $v_1  (t)= v_2  (t)$ for $ t \in [0, T_1]$. 
In particular, choosing $T_1 = 0$, $T = 1/2$  and $e_1, e_2$ to be as above, non-increasing and $e_1  \not\equiv e_2$, the corresponding $v_1, v_2$ are two distinct Leray-Hopf solutions with the same initial datum. 
\end{theoremmain}
Analysing the proof of Theorem \ref{thm:main} one realises that choosing an infinite family of non-increasing energy profiles $\{e_\alpha\}_{\alpha \in A}$ with a common $C^1$ bound, one can produce infinitely many distinct Leray-Hopf solutions with the same initial datum.

\subsubsection{Non-uniqueness of distributional solutions} 
If we drop the ambition to control the energy and require only to pre-determine the profile of the kinetic part of the energy  $\int_\Td  |v|^2 (t)$, then we produce non-unique solutions for exponents below the scaling-critical one, i.e.\ for $q<\frac{3d+2}{d+2}$. Moreover, they enjoy the regularity $v \in C(L^2) \cap C (W^{1,r})$ for any $r<\frac{2d}{d+2}$. This is our second result.

\begin{theoremmain}\label{thm:two}
Consider \eqref{eq:pnse}, \eqref{eq:orl} on $\Td  \times (0,1)$. Let $q< \frac{3d+2}{d+2}$. Fix any $e \in C^\infty([0,1]; [1/2,1])$ and $r \in (\max\{1, q-1\}, \frac{2d}{d+2})$. There exists null-mean $v \in C([0,1]; L^2(\T^d)) \cap C ([0,1]; W^{1,r}(\T^d))$ such that
\begin{enumerate}[1)]
\item $v$ solves \eqref{eq:pnse} distributionally, i.e.
\[
\int_0^1 \int_\Td -v \cdot \pat \varphi - v\otimes v  \nabla \varphi   + \A(\D) \nabla \varphi = 0, \qquad  \forall_{t \in [0,1]} \quad \int_\Td v (t) \cdot \nabla \psi = 0
\]
for any divergence-free $\varphi \in C^1 (\Td  \times [0,1])$ vanishing at $t=0$ and $t=1$, and any  $\psi \in C^1 (\Td)$;
\item the kinetic energy equals $e$, i.e.\
\begin{equation}\label{eq:epr2}
\int_\Td  |v|^2 (t) = e(t).
\end{equation}
\end{enumerate}
Moreover, fix $0 \leq T_1 < T \leq 1$ and  two energy profiles $e_1, e_2$ as above, such that $e_1 (t)= e_2 (t) $ for $t \in [0, T]$. There exists $v_1, v_2$ satisfying 1), 2) and such that $v_1  (t)= v_2  (t)$ for $ t \in [0, T_1]$. 
In particular, choosing $T_1 = 0, T =1/2$ and $e_1, e_2$ to be as above, non-increasing and $e_1  \not\equiv e_2$, the corresponding $v_1, v_2$ are two distinct distributional solutions, which belong to $C (L^2) \cap C (W^{1,r})$, dissipate the kinetic energy, and share the same initial datum.
\end{theoremmain}

%
%

%
%

\subsubsection{Existence of multiple solutions for any $L^2$ data} In Theorems \ref{thm:main}, \ref{thm:two} the initial data are attained strongly (in particular we can add  initial values to the distributional formulas for solutions, extending test functions to non-vanishing ones at $t=0$), but they are constructed in the convex integration scheme, thus possibly non-generic. This issue is addressed in our third theorem. It shows existence of energy solution emanating from {\emph any} solenoidal vector field in $L^2$, for power laws below the scaling exponent. 
\begin{theoremmain}
\label{thm:exist}
Consider \eqref{eq:pnse}, \eqref{eq:orl} on $\Td  \times (0,1)$. Let $q< \frac{3d+2}{d+2}$, $r \in (\max\{1, q-1\}, \frac{2d}{d+2})$. Fix an arbitrary nonzero $v_0 \in L^2(\T^d)$, $\divv v_0 = 0$. 
There exist continuum of $v \in C((0,1]; L^2(\T^d)) \cap L^r ((0,1); W^{1,r}(\T^d))$ such that
\begin{enumerate}[1)]
\item $v$ solves \eqref{eq:pnse} distributionally, i.e.
\[
\int_0^1 \int_\Td -v \cdot \pat \varphi - v\otimes v  \nabla \varphi   + \A(\D) \nabla \varphi = 0, \qquad  \forall_{t \in [0,1]} \quad \int_\Td v (t) \cdot \nabla \psi = 0
\]
for any divergence-free $\varphi \in C^1 (\Td  \times [0,1])$ vanishing at $t=0$ and $t=1$, and any  $\psi \in C^1 (\Td)$;
\item $v|_{t=0} = v_0$, in the sense that as $t \to 0^+$, $v(t) \to v_0$ weakly in $L^2$ and strongly in $L^{q_0}$, any $q_0 <2$.
\end{enumerate}
\end{theoremmain}

\subsection{Differences between our non-uniqueness and existence results}
The non-uniqueness Theorems \ref{thm:main}, \ref{thm:two} focus on possibly strongest notions of solutions: they allow, respectively, for full- or kinetic energy inequality and strong attainment of a (constructed) initial datum, but they do not produce non-unique solutions for any initial datum. Conversely, Theorem \ref{thm:exist} provides existence of many weak solutions for an arbitrary solenoidal initial datum in $L^2$. In particular, this is the first existence proof for the case of $q \le \frac{2d}{d+2}$. The obtained solutions are, however, much weaker than that of Theorems \ref{thm:main}, \ref{thm:two}: they do not allow for any kind of energy inequality (in fact, even their kinetic energies are in a sense pathologically large) and the initial datum is attained merely in a weak sense.

\subsection{The $3$d Navier-Stokes case}
Theorems \ref{thm:two}, \ref{thm:exist} cover also the case of non-unique weak solutions of three-dimensional Navier-Stokes equations, first proven in \cite{BV19}. Our Theorem \ref{thm:two} shows that $\nabla v \in L^{6/5-}$. This probably holds for solutions constructed in  \cite{BV19} as well, though the best regularity claimed there is $\curl v \in L^1$. Theorem \ref{thm:exist} produces infinitely many weak solutions for any divergence-free datum in $L^2$ (but with unnaturally high energies).


%
%

\subsection{Methodology and plan} Our approach follows the convex integration methods introduced to inviscid fluid dynamics in \cite{DLS09, DLS13}, culminating in \cite{Ise16, BDLSV18}, and extended to the Navier-Stokes case in the important paper \cite{BV19}. Results on a system involving fractional laplacian
  can be found in \cite{ColDLDR18, DR20, TitLuo20}. Other related interesting results include \cite{CheShv14sima, CheLuo19, CheLuo20, CheDai19, BucVic19ems, BeeBucVic20, BucShkVic19}.

We stay close to the concentration-oscillation method developed for the transport equation in \cite{ModSze18}, \cite{modena-szekelyhidi18}, and localised to avoid dimension loss in \cite{ModSat20}, see also \cite{BCDL}. 

The basic picture of the construction, as in any convex integration scheme applied to the equations of fluid dynamics, is the following. 
Given an exact flow $(v, \pi)$, i.e.\ a solution to  \eqref{eq:pnse}, one tries to distinguish the good (`laminar', `averaged') component of $v$, i.e.\  $\langle v \rangle$ and the remainder, thought to be responsible for turbulence (interestingly, the case $q=3$ in  \eqref{eq:orl}, where scaling \eqref{eq:sca} fails, is the Smagorinsky model for turbulence). A typical averaging process $\langle \cdot \rangle$ does not commute with nonlinear quantities, thus applying  $\langle \cdot \rangle$ to  \eqref{eq:pnse} yields for $u = \langle  v\rangle$
\[
\pat u + \divv (u\otimes u)  - \divv \A (\Du) +\nabla \langle \pi \rangle= \divv \big(u\otimes u -  \langle v \otimes v \rangle\big) - \divv \big( \A (\Du) -  \langle \A (\D) \rangle \big) =: \divv R.
\]
Above, $u$ is a well-behaved flow and the Reynolds stress $R$ encodes the difference between $u = \langle  v\rangle$ and the exact $v$ itself. The rough idea behind producing non-unique solutions to \eqref{eq:pnse} is to reverse-engineer the above picture. We can thus consider the following relaxation of \eqref{eq:pnse}

\begin{equation}\label{eq:pR-intro}
\begin{aligned}
\pat u + \divv (u\otimes u)  - \divv \A (\Du) +\nabla \pi&= -\divv R, \\
\divv u &= 0.
\end{aligned}
\end{equation}

Assume we have identity \eqref{eq:pR-intro} with certain $(u_0,\pi_0,R_0)$. It is easy to find at least one smooth solution of \eqref{eq:pR},
since $R_0$ is at our disposal. If one can produce another $u_1, q_1$ such that $(u_1,q_1,R_1)$ solves \eqref{eq:pR-intro} and $R_1$ is strictly smaller than $R_0$, there is a hope to iteratively diminish the Reynolds part $R_n$ to $0$ with $n \to \infty$. Consequently, in the limit one produces an exact solution $v,\pi$. Non uniqueness in the above procedure may be specified in at least two ways: 
\begin{itemize}
\item either by enforcing $v$ to be equal at some times, say for $t \in [0,1/3]$, to a given regular solution $v_1$ and for $t \in [2/3,1]$ to another regular solution $v_2$, as  for instance in \cite{BCV19} or \cite{ModSze18}.
\item or by specifying a kinetic energy profile, see e.g.\ \cite{BV19}, \cite{DLS13}, or the present work. 
\end{itemize}

\subsection{Organisation of proofs}
In Section \ref{sec:iterS}, we state the main proposition of  the paper, i.e. Proposition \ref{prop:main}, which contains the inductive step described above, from $(u_0, \pi_0, R_0)$ to $(u_1, \pi_1, R_1)$, with $R_1$ ``much smaller'' than $R_0$.  Section \ref{sec:pre} gathers preliminary material. 
In Section \ref{sec:mik} we introduce a generalisation of Mikado flows that serves as a building block for $u_1$ given $u_0$. Next, in Section \ref{sec:step}, assuming a solution $(u_0, \pi_0, R_0)$ to \eqref{eq:pR-intro} is given, we define $(u_1, \pi_1, R_1)$. 
Estimates for $(u_1-u_0)$ and $R_1$ occupy Section \ref{sec:est}. Section \ref{sec:proppf} concludes the proof of the main Proposition \ref{prop:main}. Having it in hand, we prove Theorem \ref{thm:main} in Section \ref{sec:pfT1}. The proofs of Theorems \ref{thm:two}-\ref{thm:exist} follow similar lines and therefore are only sketched in Sections \ref{sec:pfT2}-\ref{sec:pfT3}.


\subsection{Notation}
We use mostly standard notation, e.g.\ $\Td$ denotes the $d$-dimensional torus $[0, 1]^d$, $\dot W^{1,q}$ is a homogenous Sobolev space, $C_0^\infty(\T^d; B)$ are smooth functions with mean zero, domain $\T^d$ and values in set $B$ (the target set will be sometimes omitted). We take $\NN=\{1,2, \dots\}$.

We suppress the variables and the spatial domain of integration, if no confusion arises. We use $| \cdot |_.$ instead of $\| \cdot \|_.$ for norms. For $L^p$-norms on the torus $\Td$, we will abbreviate $|\cdot|_{L^p(\Td)}$ to $| \cdot |_{L^p}$ or even to $| \cdot |_{p}$. In other cases, e.g. when taking the $L^p$-norm on $\Rd$, we will explicitly write the underlying domain, where the norm is calculated, e.g. $|\cdot|_{L^p(\Rd)}$. The finite-dimensional norm is  $| \cdot |$. The projection onto null-mean functions is $\Po f:= f - \dashint_\Td f$. 

We will call $d \times d$ (symmetric) matrices \emph{(symmetric) tensor}. For a tensor $T$, we denote its traceless part by $\mathring {T} := T - \frac{1}{d} \tr(T) \Id$. The space of symmetric tensors will be denoted by $\mathcal{S}$, its open subset of positive definite tensors by $\mathcal{S}_+$. If $R$ is a symmetric tensor, $\divv R$ is the usual row-wise divergence. 


We use two types of constants $M$'s, which are uniform over iterations, and $C$'s which are not (both possibly with subscripts), for details see Section \ref{ssec:con}. All constants may vary between lines.

Further notation is introduced locally when needed.

\section{Main proposition: an iteration step}\label{sec:iterS}
Recall that $\mathcal{S}$ is the space of symmetric tensors. 
\begin{defi}
\label{def:nnr}
A solution to {\em{the Non-Newtonian-Reynolds system}} is a triple $(u,\pi,R)$
where
\begin{equation*}
u \in C^\infty([0,1] \times \T^d; \RR^d), \quad \pi \in C([0,1] \times \T^d; \RR) \quad R \in C([0,1] \times \T^d; \mathcal S)
\end{equation*}
with spatial null-mean $u,\pi$, satisfying 
\begin{equation}\label{eq:pR}
\begin{aligned}
\pat u + \divv (u\otimes u)  - \divv \A (\Du) +\nabla \pi&= -\divv \cR, \\
\divv u &= 0.
\end{aligned}
\end{equation}
in the sense of distributions. 
\end{defi}

\begin{remark}\label{rem:nonsm}
Despite smoothness of $u$, we can not require that \eqref{eq:pR} is satisfied in the classical sense or $\pi,R$ are smooth (in space), because of non-smoothness of $\A(Du)$.
\end{remark}

\begin{remark}[$R$ vs $\cR$]
Use of the trace-free Reynolds stress simplifies computation, in particular proof of the energy iterate Proposition \ref{prop:ei}. The difference between $R$ to $\cR$ is facilitated by the ambiguity of pressure: $(u,\pi,R)$ solves  \eqref{eq:pR}  $\iff (u,\pi-\frac{1}{d} tr R,\cR)$ solves \eqref{eq:pR}. 
\end{remark}

As observed in the introduction, the crucial point in the convex integration scheme is, given $(u_0, q_0, R_0)$, to produce an appropriate correction $(u_1,q_1,R_1)$ which decreases $R_i$, improves the energy gap, and retains as much regularity as possible. This single iteration step is given by 
%
\begin{prop}\label{prop:main} 
Let $\nu_0, \nu_1 \geq 0$ and $q< \frac{2d}{d+2}$ be fixed. Fix an arbitrary $e \in C^\infty([0,1]; [\frac{1}{2},1])$. There exist a constant $M$ such that the following holds. 

Let $(u_0, \pi_0, R_0)$ be a solution to the Non-Newtonian-Reynolds system \eqref{eq:pR}, as in Definition \ref{def:nnr}. Let us choose any $\delta, \eta, \epsilon \in (0,1]$.
Assume that
\begin{equation}\label{eq:ass_e}
\frac{3}{4} \delta e(t) \le  e(t) -  \Big( \int_\Td |u_0|^2 (t) + 2 \int_0^t \int_\Td  \A (D u_0) D u_0 \Big) \le \frac{5}{4} \delta e(t)
\end{equation}
and
\begin{equation}\label{eq:Rdelta}
|\cR_0 (t)|_{L^1} \le \frac{\delta}{2^7 d}.
\end{equation}
Then, there is another solution $(u_1, \pi_1, R_1)$ to \eqref{eq:pR} (as in Definition \ref{def:nnr}) such that
\begin{subequations}\label{eq:mp_all}
\begin{equation}\label{eq:mp_up}
|(u_1 - u_0)(t)|_{L^2} \le M \delta^\frac12 
\end{equation} 
\begin{equation}\label{eq:mp_uwp}
|(u_1 - u_0)(t) |_{W^{1,q}} \le \eta
\end{equation}
\begin{equation}\label{eq:mp_R}
|R_1 (t)|_{L^1} \le \eta.
\end{equation}
\end{subequations}
Furthermore
\begin{equation}\label{eq:e_contr}
\frac{3}{8} \delta e(t) \le  e(t) - \Big( \int_\Td |u_1|^2 (t) + 2 \int_0^t \int_\Td  \A (D u_1) D u_1 \Big) \le \frac{5}{8} \delta e(t).
\end{equation}
\end{prop}
\section{Preliminaries}\label{sec:pre}

\subsection{Control of $\A$}
We collect the needed growth estimates for $\A(Q)$ and for $\A(Q)Q$. 
%
%

\begin{lem}[Growth estimates for $\A$]\label{lem:gA}
Let $\A := (\nu_0 + \nu_1 |Q|)^{q-2}Q$, with $\nu_0, \nu_1 \geq 0$. Then
\begin{equation}
\label{eq:a-est-1}
|\A(Q) - \A(P)| \le \left\{ 
\begin{aligned}
 & C_{\nu_1} |Q-P|^{q-1} &\text{ for } \nu_0=0, q \leq 2 \\
 & C_{\nu_0} |Q-P| &\text{ for } \nu_0>0, q \leq 2 \\
 & C_{q, \nu_0, \nu_1}  |Q-P| \left(1 + |Q|^{q-2} + |P|^{q-2} \right) &\text{ for } q\ge2
\end{aligned}
\right.
\end{equation}
\begin{equation}
\label{eq:a-est-3}
|\A(Q)Q - \A(P)P| \leq C_{q, \nu_0, \nu_1} \big(1 + |Q|^{q-1} + |P|^{q-1} \big) |Q-P|.
\end{equation}
\end{lem}
The proof is standard. For convenience of the reader, we added it in Appendix.

\begin{remark}\label{rem:orl}
Lemma \ref{lem:gA} extends to other tensors $\A$, e.g.\ $(\nu_0+\nu_1 |Q|^2)^\frac{q-2}{2} Q$, or to ones given by an appropriate $N$-function. 
Consequently, our result extends to such tensors.
\end{remark}

\subsection{Nash-type decomposition}
Let us denote the set of positive-definite $d \times d$ symmetric tensors by $\cS_+$. We recall Lemma 2.4 in  \cite{DanSze17}
\begin{lem}\label{lem:NasUFS}
For any compact set $\cN \subset \cS_+$ there exists a finite set $K \subset \Zd$ and smooth functions $\Gamma_k: \cN \to [0,1]$, such that any $R \in \cN$ has the following representation:
\[
R = \sum_{k \in K} \Gamma^2_k (R) k \otimes k.
\]
\end{lem}

\subsection{The role of oscillations}
\label{sec:invdiv}

The convex integration paradigm is to use fast oscillations of corrector functions (correcting $u_i$ to $u_{i+1}$ in our case, roughly speaking) to inductively diminish error terms (in our case Reynolds stresses $R_i$). Thus for a function $f$ and $\lambda \in \NN$ let us define
$$f_\lambda (x) := f (\lambda x).$$ 
Observe that $f_\lambda$ has the same $L^p$ norms as $f$ since we work on $\Td$, and a factor $\lambda$ appears for each derivative, i.e.\
\[
|\nabla^s f_\lambda|_p = \lambda^s |f|_p, \quad s \in \NN \cup \{0\}.
\]
It holds
\begin{prop}[Mean value] \label{prop:mosc}
Let $a \in C^\infty(\T^d; \RR)$, $v \in C_0^\infty(\T^d; \RR)$. Then for any $r \in [1, \infty]$
\begin{equation}\label{eq:mosc}
\Big|\int_\Td a v_\lambda \Big| \le \lambda^{-1} C_r |\nabla a|_{r} |v|_{r'} 
\end{equation}
\end{prop}
\begin{proof}
The case $r=\infty$ follows the proof of Lemma 2.6 in \cite{ModSze18}. For the case $r<\infty$, since $v$ is null-mean, let us solve the Laplace equation $\divv \nabla h = v$ and define $G:= \nabla h$. It holds $(\divv G)_\lambda = \lambda^{-1} \divv (G_\lambda)$  and thus, integrating by parts and using H\"older
\[
\Big|\int_\Td a v_\lambda \Big| = \lambda^{-1} \Big|\int_\Td a \divv G_\lambda  \Big| \le \lambda^{-1} |\nabla a|_r |G_\lambda|_{r'} = \lambda^{-1} |\nabla a|_r |G|_{r'}
\]
The Sobolev embedding for the null-mean $G$ yields $|G|_{r'} \le C |\nabla G|_{L^{\min (r',d+1)}} = C|\nabla^2 h|_{L^{\min (r',d+1)}}$. This is controlled thanks to Calder\'on-Zygmund theory by $|v|_{L^{\min (r',d+1)}}$.
\end{proof}
%

Even when the l.h.s.\ of \eqref{eq:mosc} is replaced with $\int_\Td |a v_\lambda|$, the decorrelation between frequencies of $a$ and $v_\lambda$ allows to improve the generic H\"older inequality to (for the proof cf.\ Lemma 2.1 of \cite{ModSze18}):

\begin{prop}[Improved H\"older]
\label{lem:improved-holder}
Let $f,g$ be smooth maps on $\T^d$. Let $r \in [1,\infty]$. Then
\begin{equation}\label{eq:imprH}
|f g_\lambda|_r \le |f|_r |g|_r + C_r \lambda^{-\frac{1}{r}} |f|_{C^1} |g|_r.
\end{equation}
\end{prop}

\subsection{Antidivergence operators}
We provide now various inverse divergence operators, needed for construction of $R_1$ in Proposition \ref{prop:main}, with appropriate estimates. The purpose of the bilinear inverse divergences below is to extract oscillations of one function, say $g_\lambda$, out of the product $fg_\lambda$. The last of them, $\idivn^2_N$,
is an operator with symmetric tensor values, such that $\divv \divv \idivn^2_N f = f$ for every null-mean real function $f$; it facilitates construction of the $R_{lin}$ term of $R_1$, cf.\ \eqref{eq:defRlin}. 
\begin{prop}\label{prop:invdiv}
Let $p,r,s \in [1,\infty]$ and $\frac{1}{p} =  \frac{1}{s} +  \frac{1}{r}$.

\begin{enumerate}[(i)]
\item ($ \divn$: symmetric antidivergence)  There exists
$\divn: C_0^\infty(\T^d; \RR^d) \to C_0^\infty(\T^d; \cS)$
such that
$\divv \divn u = u$ and for $i\ge0$ one has
\begin{equation}\label{eq:inv_stnd}
| \nabla^i \divn u|_{p} \le  C_{k,p} |\nabla^i u|_{{p}} ,
\end{equation}
and for the fast oscillating $u_\lambda$
\begin{equation}\label{eq:inv_stnd2}
| \nabla^i \divn u_\lambda|_{p} \le C_{k,p} \lambda^{i-1} |\nabla^i u|_{p}  .
\end{equation}

\item ($ \idivn_N$: improved symmetric bilinear antidivergence) For any $N \ge1$ there exists a bilinear operator $\idivn_N: C^\infty(\T^d; \RR) \times C^\infty_0(\T^d; \RR^d)   \to C^\infty_0(\T^d; \cS)$ such that  $\divv \idivn_N (f, u) = f u - \dashint fu$ and 
\begin{equation}\label{eq:inv_imprA}
| \mathcal{R}_N (f, u_\lambda) |_{p} \leq C_{d,p,s,r,N}  | u |_{s} \Big(\frac{1}{\lambda} |f|_{r} + \frac{1}{\lambda^N} |\nabla^N f|_{r} \Big) .
\end{equation}

\item ($ \tilde \idivn_N$: improved symmetric bilinear antidivergence on tensors)
For any $N \ge1$ there exists a bilinear operator $\tilde \idivn_N: C^\infty(\T^d; \RR^d) \times  C^\infty_0(\T^d; \RR^{d \times d}) \to C^\infty_0(\T^d; \cS)$ such that  $\divv \tilde \idivn_N (v,T) = Tv - \dashint Tv$ and 
\begin{equation}\label{eq:inv_imprAt}
| \tilde \idivn_N (v, T_\lambda) |_{p} \leq  C_{d,p,s,r,N} |T|_{s} \Big( \frac{1}{\lambda} |v|_{r} + \frac{1}{\lambda^N} |\nabla^N v|_{r} \Big) .
\end{equation}

\item ($ \idivn^2_N$: improved symmetric bilinear double antidivergence) For any $N \ge1$ there exists a bilinear operator 
$\idivn^2_N: C^\infty(\T^d; \RR) \times C^\infty_0(\T^d; \RR)   \to C^\infty_0(\T^d; \cS)
$ such that $\divv \divv \mathcal{R}^2_N (f,g) = fg - \dashint fg$ 
and for any $j \in \NN \cup\{0\}$
\begin{equation}\label{eq:inv_impr2div}
|\nabla^{j} \idivn^2_N (f, g_\lambda)|_p  \le C_{j,d,p,s,r,N} \lambda^j |g|_{W^{j,s}} \Big( \frac{1}{\lambda^{2}} |f|_{r} + \frac{1}{\lambda^{N}} |\nabla^N f|_{r}  + \frac{1}{\lambda^{2N+j}} |\nabla^{2N +j} f|_{r} \Big).
\end{equation}

\end{enumerate}
\end{prop}
The proof is standard, cf. \cite{DLS13, ModSze18, ModSat20} and can be found in Appendix.

\begin{remark}[$\idivn^2_N \neq \idivn_N \circ \idivn_N$]
For the operator defined in (iv), we use the  notation $\idivn^2_N$ to denote that this operator acts as a double antidivergence. It does not coincide in general with $\idivn_N \circ \idivn_N$. 
\end{remark}

\begin{remark}[$\idivn_\infty$]\label{rem:ia}
The above bilinear antidivergences may be thought of as approximations of `ideal antidivergence' operators $\mathcal{R}_\infty$, $\mathcal{R}^2_\infty$ satisfying
\begin{equation}
\label{eq:r-infty}
| \mathcal{R}_\infty (f, u_\lambda) |_{p} \lesssim  \frac{1}{\lambda} | u |_{s}  |f|_r, \qquad | \nabla^j \mathcal{R}^2_\infty (f, g_\lambda) |_{p} \lesssim \lambda^{j-2} |g |_{W^{j,s}}  |f|_r,
\end{equation}
where the gap between $\idivn_N$ and $\idivn_\infty$ closes as $N \to \infty$, similarly for $\idivn^2_N$ and $\idivn^2_\infty$. 
\end{remark}

\section{Mikado flows}\label{sec:mik}

In this section we introduce the building blocks of our construction, namely the \emph{concentrated localized traveling Mikado flows}, a generalization of the \emph{Mikado flows} of \cite{DanSze17}. 

The original Mikado flows  of \cite{DanSze17} are fast oscillating pressureless stationary solutions to Euler equations having the form
\begin{equation}
\label{eq:mikado-dan-sze}
  \Psi_\lambda^k(x) k = \Psi^k(\lambda x) k,
\end{equation}
where $k \in \Zd$ is a direction. For a finite set of directions $K$ (given by the decomposition Lemma \ref{lem:NasUFS}) one can choose functions $\Psi^k \in C^\infty_0(\Td, \RR)$ so that the following holds for any $k, k' \in K$ and $\lambda \in \N$
\begin{equation}\label{sMikp}
\begin{aligned}
(i)\qquad& \divv   \Psi_\lambda^k k= 0, \\
(ii)\qquad& \divv ( \Psi_\lambda^k k \otimes  \Psi_\lambda^k k) = 0, \\
(iii)\qquad&\dashint_\Td (\Psi^k)^2 = 1 \quad \text{ thus } \quad\dashint_\Td \Psi_\lambda^k k \otimes  \Psi_\lambda^k k = k \otimes k,\\
(iv)\qquad&  \Psi_\lambda^k k \quad\text{ and } \quad\Psi_\lambda^{k'} k' \quad\text{ have disjoint supports for }k \neq k'.
\end{aligned}
\end{equation}
Satisfying property (i) is equivalent to choosing $\Psi^k$ so that    $\nabla \Psi^k \cdot k \equiv 0$, then also (ii) follows. Having (iii) is a normalisation of $\dashint_\Td (\Psi^k)^2$. Disjointness of supports (iv) is ensured in $d\ge3$ via an appropriate choice of an anchor point $\zeta_k$ for the cylinder $B_\rho(0) + \{ \zeta_k + s k \}_{s \in \RR} + \Zd$ (which is the periodisation of the cylinder $B_\rho(0) + \{ \zeta_k + s k\}_{s \in \RR}$  with radius $\rho$ and axis being the line passing through $\zeta_k$ with direction $k$). Such choice is possible in view of
\begin{lem}[Disjoint periodic tubes]
\label{l:disj-supp-dmin1}
Let $d \geq 3$. Then there exist $\{\zeta_k\}_{k \in K}$ and $\rho>0$ such that
\begin{equation}
\label{eq:disj-supp-d1}
\big( B_\rho(0) + \big\{ \zeta_k + s k\big\}_{s \in \RR} + \Zd  \big) \cap \big( B_\rho(0) + \big\{ \zeta_{k'} + s' k'\}_{s' \in \RR} + \Zd  \big) = \emptyset
\end{equation}
for all $k,k' \in K$, $k \neq k'$. 
\end{lem}
Proof can be found in Appendix.
%

The convex integration approach uses the properties (i)-(iv) to diminish a given 
Reynolds stress $R_0$ of a given solution $(u_0, \pi_0, R_0)$ to the Non-Newtonian-Reynolds system \eqref{eq:pR} by correcting $u_0$ roughly as follows. Thanks to (iii), we can decompose $R_0$ via the Nash Lemma \ref{lem:NasUFS} into 
\begin{equation}\label{eq:smb}
\sum_{k} \Gamma^2_k (R_0) k \otimes k =  \sum_{k} \Gamma^2_k (R_0) \dashint_\Td \Psi_\lambda^k k \otimes  \Psi_\lambda^k k.
\end{equation}
Let us add to $u_0$ the corrector $\tilde u = \Gamma_k (R_0) \Psi_\lambda^k k$. Recall the notation $\Po f:= f - \dashint_\Td f$. Thanks to (iv) and (ii)
\begin{equation*}
\divv \big( \tilde u \otimes \tilde u - R_0 \big)=  \sum_{k} \Po \left(\Psi_\lambda^k k \otimes  \Psi_\lambda^k k \right) \nabla \Gamma^2_k (R_0) .
\end{equation*}
Since the term $\Po \left(\Psi_\lambda^k k \otimes  \Psi_\lambda^k k \right)$ is $\lambda$-periodic and null mean, applying  $\mathcal{R}_\infty$ of \eqref{eq:r-infty} to the r.h.s.\ above yields $R_1$ of order $\lambda^{-1}$, such that $\divv R_1 = \divv (\tilde u \otimes \tilde u - R_0)$. So picking $\lambda$ large, i.e.\ letting $\Psi_\lambda^k$ oscillate fast, allows to deal with the error $R_0$. The property (i) allows to control $\divv\tilde u$.

\subsection{Concentrated Mikado flows}\label{ssec:cmf}
Since 
\begin{equation*}
|\nabla^i \Psi^k_\lambda k|_{L^p(\Td)} = \lambda^i  |\nabla^i \Psi^k k|_{L^p(\Td)}=C \lambda^i \to \infty \quad \text{ as } \quad \lambda \to \infty,
\end{equation*}
fast oscillations, in general, blow up derivatives of the corrector $\tilde u$. Thus controlling Sobolev norms of velocity fields appearing over convex integration steps seems problematic. This issue may be circumvented by a \emph{concentration} mechanism, introduced in  \cite{ModSze18} and critically inspired by \cite{BV19}.

Let us briefly explain it. For $n \le d$, take a compactly supported smooth function $f: \RR^n \to \RR$, rescale it to $f_\mu (x) = \mu^a f(\mu x)$, $\mu \ge 1$, and periodize without renaming to $f_\mu: \T^d \to \RR$. This is \emph{concentrating} and results in 
\begin{equation*}
|\nabla^i f_\mu|_{L^p(\T^d)} = \mu^{a+i - \frac{n}{p}} |\nabla^i f|_{L^p(\RR^n)} =C \mu^{a+i - \frac{n}{p}} \quad \text{ and } \quad |\nabla^i f_{\mu,\lambda}|_{L^p(\T^n)}  = C\lambda^i\mu^{a+i - \frac{n}{p}}.
\end{equation*}
%
This procedure yields the `concentrated Mikado' $\Psi_\mu^k k$ satisfying 
\[
|\nabla^i \Psi_\mu^k  k |_{L^p(\T^d)}  = C \lambda^i \mu^{a+i-\frac{n}{p}}.
\]
Having now an interplay between $\lambda$ and $\mu$ one can expect to control certain Sobolev norms by choosing $a,p$ appropriately. However, to preserve the properties (i) and (ii) of \eqref{sMikp}, i.e.\ $\nabla f_\mu \cdot k \equiv 0$ (or in other words: $\Psi_\mu^k k$ being the Euler flow), the underlying function $f_\mu$ cannot depend on the direction $k$. It means that the underlying real function is not compactly supported in $\RR^d$, but at best in $\RR^{d-1}$. Thus at best $n=d-1$, but then
\[
|\nabla^i \Psi_\mu^k  k |_{L^p(\T^d)}  = C \lambda^i \mu^{a+i-\frac{d-1}{p}}.
\]
The quantity $\dashint_\Td  \Psi_\mu^k k \otimes \Psi_\mu^k k$ shall be of order $k \otimes k$, cf.\ \eqref{eq:smb}. Therefore $\dashint  |\Psi_\mu^k|^2$ should be $\lambda$- and $\mu$-independent, cf.\ (iii) of \eqref{sMikp}. This leads to the choice $a=\frac{d-1}{2}$ above and consequently to 
\begin{equation}\label{eq:cmfb}
(v)\qquad \qquad |\nabla^i \Psi_\mu^k  k |_{L^p(\T^d)}  = C \lambda^i \mu^{\frac{d-1}{2}+i-\frac{d-1}{p}}.
\end{equation}
Scaling \eqref{eq:cmfb} would force us to prove our results 
with $d$ substituted by $d-1$, so that e.g.\ $q$ could vary only in the interval $1<q< \frac{2(d-1)}{(d-1)+2}$ (which in particular requires $d \geq 4$). 
%

Summing up, the concentrated Mikado $\Psi_\mu^k  k$ satisfies properties \eqref{sMikp} (i)--(iv) of the original Mikado, but has unsatisfactory scaling (v). 

\subsection{Concentrated localized Mikado flows}
A natural idea to deal with the `loss of dimension' in \eqref{eq:cmfb} is to localise the Ansatz \eqref{eq:mikado-dan-sze}. Let us thus take a smooth radial cutoff function $\phi$ and define $\Phi: \RR^d \to \RR$ via  $\Phi (x) =\phi (|x|)$. We want to retain gains stemming from concentrating, and since now  $\Phi: \RR^d \to \RR$, while $\Psi$ of \eqref{eq:mikado-dan-sze} allowed merely for $d-1$ concentrations, it is better to concentrate in $\Phi$, thus producing $\Phi_\mu$. We periodize this function without renaming it and allow to oscillate at an independent frequency $\lambda_1$. Hence our new Ansatz reads
\begin{equation}
  \Psi_{\lambda_2}^k \Phi_{\mu, \lambda_1} k.
\end{equation}
Let us now state and prove a result gathering needed properties of the cutoff $\Phi_{\mu, \lambda_1}$.
\begin{lem}
\label{l:mikado-fcn-d}
Let $K \subset \Zd$ be a fixed finite set of directions. There exists $\rho>0$ such that for every
$\lambda_1 \in \N, \mu \in \N, \mu \geq \rho^{-1}$
there is $\Phi^k_{\mu, \lambda_1} \in C^\infty(\Td; \RR)$ with the following properties
\begin{equation}\label{eq:cf1}
\dashint_{\Td} (\Phi^k_{\mu, \lambda_1})^2 dx = 1, \qquad 
|\Phi^k_{\mu, \lambda_1}|_{W^{i,r}(\Td)} \leq M_{i,r,k} \lambda_1^i \mu^{i+\frac{d}{2} - \frac{d}{r}},
\end{equation}
\begin{equation}\label{eq:cf2}
\supp \Phi^k _{\mu, \lambda_1}(\, \cdot \, - sk) \cap \supp \Phi^{k'}_{\mu, \lambda_1}(\, \cdot \, - sk') = \emptyset \quad \text{ for all } \quad k, k' \in K, k \neq k', s \in \RR.
\end{equation}
\end{lem}
\begin{proof}
Take $\Phi \in C^\infty_c(\RR^{d})$, with ${\rm supp}\, \Phi \subseteq B_1(0) \subseteq \RR^{d}$ such that $\int_{\RR^{d}} \Phi^2 = 1$. Let us concentrate $\Phi$ to 
$\Phi_\mu : \Td \to \RR$, hence 
$\dashint_{\Td} (\Phi_{\mu})^2 = \mu^{a - \frac{d}{2}}$. Choosing $a= \frac{d}{2}$ yields the desired $\dashint_{\Td} \Phi_{\mu}^2 dx = 1$ and $|\Phi_{\mu}|_{W^{i,r}(\Td)} \leq C_{i,r}  \mu^{i+\frac{d}{2} - \frac{d}{r}}$. 
Concentration gives also $\supp \Phi _{\mu} \subseteq B_{1/\mu}(0) + \Z^d$.

Let $\rho>0$ and $\{\zeta_k\}_{k \in K}$  be given by Lemma \ref{l:disj-supp-dmin1}.
We now set for every $k \in K$ and $\lambda_1 \in \N $,
\[
\Phi_{\mu, \lambda_1}^k (x) := \Phi_\mu ( \lambda_1( x- \zeta_k)). 
\]
Property \eqref{eq:cf1} now follows from the scaling properties of $\Phi_\mu$, whereas  \eqref{eq:cf2} from $\supp \Phi _{\mu} \subseteq B_{1/\mu}(0) + \Z^d$, \eqref{eq:disj-supp-d1}, and assumed $\mu \geq \rho^{-1}$
.
%
%
%
%
%
%
\end{proof}

\subsection{Concentrated localized traveling Mikado flow}
Unsurprisingly, introducing $d$-dimensional cutoff $\Phi$ destroys the properties (i) - (iii) of standard Mikados. The most severe loss, due to its critical scaling, is not having (ii) anymore. A crucial idea how to handle this issue, introduced in \cite{BV19}, is to let the cutoff function $\Phi$ travel in time along $l^k$ with speed $\omega$. This leads to a corrector term $Y^k$ (see below), whose time derivative compensates lack of (ii). At the same time $Y^k$ is of order $\frac{1}{\omega}$, so it can be controlled by choosing $\omega$ large.

The concentrated localized traveling Mikado flow is our final Ansatz. It will be denoted by $W^k$, but it is important to bear in mind that it is determined by  the parameters
\begin{equation*}
\mu, \lambda_1, \lambda_2, \omega \in \N.
\end{equation*}
The next proposition concerns our final Mikado flows $W^k$ and Mikado correctors $Y^k$.
\begin{prop}\label{prop:cltmik}
Let $K \subset \Zd$ be a fixed finite set of directions. Let $\Psi_{\lambda_2}^k$ be the function used to produce the standard Mikado \eqref{eq:mikado-dan-sze} with its properties (i) -- (iv). Let $\Phi_{\mu, \lambda_1}$ be the localisation provided by Lemma \ref{l:mikado-fcn-d}.

Define the functions $W^k: \Td \times [0,1]  \to \Rd$, $Y^k: \Td \times [0,1]  \to \Rd$ by
\begin{equation}
\label{eq:def-mikado-d}
W^k (x,t):= \Big( \Psi_{\lambda_2}^k \Phi^k_{\mu, \lambda_1} k \Big) (x -\omega t k), \qquad 
Y^k (x,t):=  \Big( \frac{1}{\omega}(\Psi^k_{\lambda_2})^2 (\Phi_{\mu, \lambda_1}^k)^2 k \Big) (x - \omega t k).
\end{equation}

There exists $\rho>0$ such that for every $\mu, \lambda_1, \lambda_2, \omega \in \N$ satisfying
\begin{equation}
\label{eq:general-assumpt-mikado}
\mu \geq \frac{1}{\rho}, \quad \frac{\lambda_2}{\lambda_1} \in \N \quad \text{ and } \quad
\frac{\lambda_1 \mu}{\lambda_2} < \frac{1}{2}
\end{equation}
the functions $W^k, Y^k$ are spatially $\lambda_1$-periodic are 
have the following properties:
\begin{equation}\label{eq:fdc}
(v') \qquad |W^k(t)|_{W^{i,r}(\Td)}   \le M_{i,r}   {\lambda_2}^i  \mu^{\frac{d}{2} - \frac{d}{r}}, \qquad
|Y^k(t)|_{W^{i,r}(\Td)}   \le M_{i,r}  \frac{ {\lambda_2}^i  \mu^{d - \frac{d}{r}}}{\omega};
\end{equation}
\begin{equation}
\label{eq:mean-value-mikado}
(iii')  \qquad  \Big| \dashint W^k(t) \otimes W^k(t) - k \otimes k \Big| \leq M_1 \frac{\lambda_1 \mu}{{\lambda_2}};
\end{equation}
\begin{equation*}
(iv') \qquad \text{ for } k, k' \in K, \; k \neq k' {\rm supp}\, W^k \cap {\rm supp}\, W^{k'}= \emptyset;
\end{equation*}
\begin{equation}\label{eq:fdmcut:5}
(ii') \qquad \partial_t Y^k + \divv (W^k \otimes W^k) = 0.
\end{equation}
\end{prop}
\begin{proof}
The spatial $\lambda_1$-periodicity of $W^k, Y^k$ follows from the assumption ${\lambda_2}/\lambda_1 \in \N$.
Since $W^k, Y^k$ are obtained from stationary functions by means of a Galilean shift, for \eqref{eq:fdc} and \eqref{eq:mean-value-mikado} it suffices to estimate the respective stationary functions.
\begin{equation}\label{eq:msii}
|\Psi_{\lambda_2}^k \Phi^k_{\mu, \lambda_1} k |_{W^{i,r}(\T^d)}
\leq M_k \sum_{j=0}^i  |\Psi_{\lambda_2}|_{W^{i-j,\infty}(\Td)} |\Phi_{\mu,\lambda_1}|_{W^{j,r}(\Td)} \le  M_{i,r,k} \sum_{j=0}^i \lambda_1^j \mu^{j + \frac{d}{2} - \frac{d}{r}} {\lambda_2}^{i-j}
\end{equation}
with the second inequality due to \eqref{eq:mikado-dan-sze} and \eqref{eq:cf1}. Since by assumption $\lambda_1 \mu < {\lambda_2}$, we obtain \eqref{eq:fdc} estimate for $W^k$.
A similar computation yields the estimate for $Y^k$. For  \eqref{eq:mean-value-mikado} we compute

\[
\Big| \dashint (\Psi^k_{\lambda_2})^2 (\Phi_{\mu, \lambda_1}^k)^2(k \otimes k) - (k \otimes k) \Big| = \Big|  (k \otimes k) \dashint (\Phi_{\mu, \lambda_1}^k)^2 \left( (\Psi^k_{\lambda_2})^2 - 1 \right) \Big|  =:I
 \]
 with the equality valid because the normalisation of \eqref{eq:cf1} holds. Since the normalisation (iii) of the standard Mikado $\Psi^k$ implies that $(\Psi^k_{\lambda_2})^2 - 1$ is null-mean, and it oscillates at the frequency $\lambda_2$, by Proposition \ref{prop:mosc} we have
\[
I  \le  \frac{M_k}{{\lambda_2}} \left|\nabla \left( (\Phi_{\mu, \lambda_1}^k)^2 \right) \right|_{1} |(\Psi^k)^2-1|_{\infty} \le M_k \frac{\lambda_1 \mu}{{\lambda_2}}
\]
via \eqref{eq:cf1}. We reached \eqref{eq:mean-value-mikado}.

Disjointness of supports of follows from
\eqref{eq:cf2}: Assume that for some $(x,t) \in \Td \times [0,1]$ and $k,k' \in K$, $k \neq k'$, $W^k(x,t) W^{k'}(x,t) \neq 0$.
Hence in view of the definition \eqref{eq:def-mikado-d}
\begin{equation*}
x - \omega t k \in {\rm supp}\, \Phi_{\mu, \lambda_1}^k, \quad x - \omega t k' \in {\rm supp}\, \Phi_{\mu, \lambda_1}^{k'},
\end{equation*}
thus contradicting \eqref{eq:cf2}.

The Mikado functions have the form 
\begin{equation*}
W^k(x,t) = F(x - \omega t k) k, \quad Y^k (x,t) = \frac{1}{\omega} G(x - \omega t k) k,
\end{equation*}
with $F^2 = G$. Therefore
$\divv (W^k \otimes W^k) =  \big( \nabla G(x - \omega t k) \cdot k \big) k$,
whereas $\partial_t Y^k= - \big(\nabla G(x - \omega t k) \cdot k \big) k$. Hence \eqref{eq:fdmcut:5}.
\end{proof}

\begin{remark}
\label{rmk:comparison-mikado}
Let us compare our $W^k$ 
with the concentrated Mikado.  
\begin{enumerate}[(a)]
\item $W^k$ is not divergence free, i.e.\ (i) does not hold. Furthermore $W^k$ is now time dependent.
\item  $W^k$ does not satisfy (ii). There appears Mikado corrector $Y^k$ to compensates this deficiency, see \eqref{eq:fdmcut:5}. This means however that the new term $Y^k$ must be appropriately estimated.
\item Property (iii) holds approximately, see \eqref{eq:mean-value-mikado}.
\item Supports are pairwise disjoint (now in space-time).
\item The scaling with a dimension loss  \eqref{eq:cmfb} is now improved to \eqref{eq:fdc}, which is our main gain.
\end{enumerate}

Proposition \ref{prop:cltmik} yields the following estimates in relation to (b), (c), (e):
\[
\begin{aligned}
&\left| Y^k(t) \right|_{L^2(\Td)} \sim \frac{\mu^{d/2}}{\omega}, \quad
&\left| \nabla Y^k(t) \right|_{L^q(\Td)} \sim \lambda_2 \frac{\mu^{d-\frac{d}{q}}}{\omega}, \\
&\Big| \dashint W^k(t) \otimes W^k(t) - k \otimes k \Big| \sim \frac{\lambda_1 \mu}{{\lambda_2}}, \quad &|\nabla W^k(t)|_{L^q(\Td)} \sim {\lambda_2}  \mu^{\frac{d}{2} - \frac{d}{q}}.
 \end{aligned}
\]
In order to deal with (a), let us recall the heuristics of an ideal antidivergence operator $\idivn_\infty$ of Remark \ref{rem:ia}. A short computation involving \eqref{eq:def-mikado-d}, \eqref{eq:fdc}, and \eqref{eq:cf1} yields
\[
\left| \idivn_\infty (\divv W^k) \right|_{L^2(\Td)}  \sim \frac{\lambda_1 \mu}{{\lambda_2}}, \qquad \left| \idivn_\infty \partial_t W^k \right|_{L^1(\Td)}  \sim \frac{\lambda_1 \mu}{{\lambda_2}} \cdot \frac{\omega}{\mu^{d/2}}
\]
Therefore, seeking smallness of
\begin{equation}\label{eq:comparison-mikado}
 \frac{\mu^{d/2}}{\omega}, 
\quad \frac{\lambda_1 \mu}{{\lambda_2}},   
 \quad  {\lambda_2}  \mu^{\frac{d}{2} - \frac{d}{q}},  \quad  \frac{\lambda_1 \mu}{{\lambda_2}} \cdot \frac{\omega}{\mu^{d/2}}
\end{equation}
will motivate the choice of the relations between the parameters $\mu, \lambda_1, \lambda_2, \omega$ in Section \ref{sec:proppf}. 

Notice that we did not add the term $ \lambda_2 \frac{\mu^{d-\frac{d}{q}}}{\omega}$ to the list \eqref{eq:comparison-mikado}, as  it is the product of the first and the third term in \eqref{eq:comparison-mikado}, and thus its smallness is implied by smallness of these terms.
\end{remark}

\section{Definition of $(u_1, \pi_1, R_1)$}\label{sec:step}
Let $(u_0, \pi_0, R_0)$ be a solution to the Non-Newtonian-Reynolds system \eqref{eq:pR}, and $\delta, \eta \in (0,1]$ as in Proposition \ref{prop:main}. 
%
%
We define
\[u_1 := u_0 + \vp + \vc, \quad \pi_1 := \pi_0 + \pi_p,\] where 
\begin{itemize}
\item $\vp$ is a perturbation based on the Mikado flows of Section \ref{sec:mik}, aimed at decreasing $R_0$, 
\item $\vc$ is a corrector restoring solenoidality of $u_1$ and compensating for our Mikado flows not solving Euler equations.
\end{itemize}
 Since  $(u_0, \pi_0, R_0)$ solves  \eqref{eq:pR}, it holds in the sense of distributions
\begin{equation}\label{eq:rey1}
\begin{aligned}
\pat u_1 +  \divv (u_1\otimes u_1)  - \divv  \A (\Do)+&\nabla \pi_1 = \\
& \pat (\vp + \vc) + \divv (u_0 \otimes \vp + \vp \otimes u_0) \\
+&\, \divv (u_0 \otimes \vc + \vc \otimes u_0 + \vp \otimes \vc + \vc \otimes \vp + \vc \otimes \vc) \\
+&\, \divv (\vp \otimes \vp - \cR_0) \\
-&\, \divv \left(\A \big(D (u_0 + \vp + \vc) \big) - \A (D u_{0}) \right) + \nabla \pi_p.
\end{aligned}
\end{equation}
\subsection{Decomposition of  $R_0$ and energy control}
In general, $R_0$ is only continuous (recall Definition \ref{def:nnr} and Remark \ref{rem:nonsm}). Since it is convenient to work with smooth objects, we regularize  $R_0$ (extended for times outside $[0,1]$ by $R_0(x,0)$ and $R_0(x,1)$, respectively) with the standard mollifier $\phi_\epsilon$ in space and time. Thus
\begin{equation}
\label{eq:regularization-r0}
\cR_0^\epsilon	 := \cR_0 * \phi_\epsilon.
\end{equation}
%
Now $R_0^\epsilon$ is smooth and \eqref{eq:Rdelta} implies
\begin{equation}
\label{eq:Rdelta-epsilon}
|\cR_0^\epsilon(t)|_1 \leq \frac{\delta}{2^7d}, \quad \text{ for every $t \in [0,1]$}.
\end{equation}
Next, we decompose $\cR_0^\epsilon$ into basic directions. 
In order to stay within $\cS_+$ of Lemma \ref{lem:NasUFS}, we shift and normalise $\cR_0^\epsilon$ via
\[
\Id + \frac{\cR_0^\epsilon (x,t)}{\varrho(x,t)}, \qquad \text{with}
\]
\begin{equation}\label{eq:motiv_rho}
\varrho(x,t) := 2 \sqrt{\epsilon^2 + |\cR_0^\epsilon (x,t)|^2} + \gamma_0(t),
\end{equation}
\begin{equation}\label{eq:gnot}
\gamma_0(t) :=\frac{e(t) (1-\frac{\delta}{2}) -  \left(  \int_\Td  |u_0|^2 (t)  + 2 \int_0^t \int_\Td  \A (D u_0) D u_0 \right)}{d}.
\end{equation}
The role of $\sqrt{\epsilon^2 + ..}$ is to avoid the degeneracy $|\cR_0(x,t)| =0$, whereas the role of $\gamma_0$ is to pump energy into the system, thus facilitating the step \eqref{eq:ass_e} $\to$  \eqref{eq:e_contr}. Observe that 
$\gamma_0>0$ because of \eqref{eq:ass_e}. 
The choice \eqref{eq:motiv_rho} yields in particular  $\varrho \ge 2|\cR^\epsilon_0(x,t)|$ and hence 
\begin{equation}\label{eq:unifN}
\frac12 \Id \le \Id + \frac{\cR_0^\epsilon}{\varrho} \le \frac32 \Id.
\end{equation}

\begin{remark}
\label{rmk:directions-fixed}
The set $\cN \subset \cS_+$ of Lemma \ref{lem:NasUFS} is fixed by \eqref{eq:unifN} uniformly over the convex integration iterations. 
\end{remark}
Define
\begin{equation}
\label{eq:def-ak}
a_k(x,t) := \varrho^\frac12 (x, t)  \Gamma_k \Big(\Id + \frac{\cR_0^\epsilon (x,t)}{\varrho(x,t)}\Big).
\end{equation}
Thanks to Lemma \ref{lem:NasUFS} it holds 
\begin{equation}\label{eq:Rp2m}
 \varrho \Id  +  \cR_0^\epsilon =  \sum_{k \in K} \varrho \Gamma^2_k \Big(\Id + \frac{\cR_0^\epsilon}{\varrho}\Big) k \otimes k  = \sum_{k \in K} a_k^2 \, k \otimes k.
\end{equation}
\subsection{Choice of $u_p$} Now we choose the principal corrector, motivated by the corrector $\tilde u$ that appeared in the initial part of Section \ref{sec:mik}. Let $W^k$ be the  Mikado flow of Proposition \ref{prop:cltmik} with $a_k$ defined by \eqref{eq:def-ak}. Let
\begin{equation}\label{eq:cvp}
\vp (x,t) : = \sum_{k \in K} a_k (x,t)  W^k (x,t).
\end{equation}
The disjoint supports of $W^k (t),  W^{k'} (t)$, $k \neq k'$
imply
\begin{equation}\label{eq:Rp1}
\vp \otimes \vp= \sum_{k \in K} a^2_k W^k \otimes  W^k.
\end{equation}
Recall the notation $\Po f:= f - \dashint_\Td f$. Use \eqref{eq:Rp2m} and \eqref{eq:Rp1} to write
\begin{equation}\label{eq:Rp3}
\begin{aligned}
\vp  \otimes \vp - \cR_0^\epsilon  &=\varrho  \Id   + \sum_{k \in K} a^2_k \left( W^k  \otimes W^k  - k \otimes k \right) \\
& = \varrho \Id + \sum_{k \in K} a_k^2 P_{\neq 0} (W^k \otimes W^k )  + a_k^2 \Big( \dashint W^k \otimes W^k - k \otimes k \Big)
\end{aligned}
\end{equation}
We therefore have
\begin{equation}\label{eq:divdl}
\begin{aligned}
\divv (\vp \otimes \vp - \cR^\epsilon_0) 
 =& \nabla \varrho + \sum_{k \in K} P_{\neq 0} (W^k \otimes W^k ) \nabla a_k^2 \\ & + \sum_{k \in K} \Big( \dashint W^k \otimes W^k - k \otimes k \Big) \nabla a_k^2  + \sum_{k \in K} a_k^2 \divv ( W^k \otimes W^k).
 \end{aligned}
\end{equation}
\begin{remark}
Observe that for the original Mikados, or their concentrated version, the second line of \eqref{eq:divdl} vanishes. These additional terms will be taken care of by \eqref{eq:mean-value-mikado} and 
\eqref{eq:fdmcut:5}.
\end{remark}
In order to avoid troublesome solenoidality correctors of the last $\Sigma$ in \eqref{eq:divdl}, let us (Helmholtz) project it onto divergence free vectors by $P_H = Id - \nabla \Delta^{-1} \divv$ and balance the identity by incorporating $\nabla \Delta^{-1} \divv$ into the pressure, with the new pressure \[\tilde \varrho:= \varrho + \Delta^{-1} \divv \sum_{k \in K}   a^2_k \, \divv \left( W^k  \otimes W^k \right).\]
Applying $\Po$ to both sides of the resulting identity, we arrive at
\begin{equation}
\label{eq:Rp3d}
\begin{aligned}
\divv (\vp \otimes \vp - \cR^\epsilon_0) 
 =& \nabla \tilde \varrho + \Po \sum_{k \in K} P_{\neq 0} (W^k \otimes W^k ) \nabla a_k^2 \\  + \,&\Po \sum_{k \in K} \Big( \dashint W^k \otimes W^k - k \otimes k \Big) \nabla a_k^2  + \Po P_H \sum_{k \in K} a_k^2 \divv ( W^k \otimes W^k).
 \end{aligned}
\end{equation}

\subsection{Choice of $u_c$}
The corrector term $\vc$ has the following roles: (i) to cancel the highest-order bad term $`\divv( W^k \otimes W^k)$' of \eqref{eq:Rp3d} via \eqref{eq:fdmcut:5}, (ii) to render the entire perturbation $u_p + u_c$ solenoidal and (iii) null-mean. 

For (i), observe that \eqref{eq:fdmcut:5} implies 
\begin{equation}\label{eq:twcancel}
\Po P_H \sum_{k \in K}   a^2_k \partial_t Y^k+ \Po P_H  \sum_{k \in K}   a^2_k \divv (W^k \otimes W^k) = 0.
\end{equation}
Thus taking
\begin{equation}\label{eq:uc1def}
u_c^I := \Po P_H \sum_{k \in K}   a^2_k Y^k
\end{equation}
will allow to cancel, with a part of the time derivative of $u_c^I$, the bad term $`\divv( W^k \otimes W^k)$' of \eqref{eq:Rp3d}.

For (ii), observe that thanks to $P_H$, $u_c^I$ is already solenoidal. Therefore it suffices to compensate lack of solenoidality of $\vp$. We will now define $u_c^{II}$ accordingly. 
By the definition \eqref{eq:cvp} of $\vp$, the definition \eqref{eq:def-mikado-d} of $W^k$, and since $\divv\Psi_{\lambda_2}^k k =0$ (cf.\ the property (i) of \eqref{sMikp}) we have
\[
\divv  \vp (x,t) = \sum_{k \in K} \divv \left( a_k (x,t)  \left( \Psi_{\lambda_2}^k \Phi^k_{\mu, \lambda_1} k \right) (x-\omega t k) \right) = 
 \sum_{k \in K}  \Psi^k_{\lambda_2} (x) \; k \cdot \nabla \big( a_k\Phi_{\mu, \lambda_1}^k(x - \omega t k) \big).
\]
Therefore we define
\begin{equation}\label{eq:precvc}
u_c^{II}(t, \cdot) := - \divv  \sum_{k \in K}  \idivn^2_N \Big( k \cdot \nabla \left( a_k (t) \Phi^k_{\mu, \lambda_1}(\, \cdot \, - \omega t k  ) \right), \; \Psi^k_{\lambda_2} \Big),
\end{equation}
where $\idivn^2_N$ is the double antidivergence given by Proposition \ref{prop:invdiv}, and $N$ will be fixed later (see the discussion at the beginning of Section \ref{sec:est}). 

Since $\divv \divv \idivn^2_N = \Id\, \Po$, $\divv u_c^{II} + \divv \vp=0$. 

As $u_c^I, u_c^{II}$ are null-mean, to take into account the condition (iii), it suffices to define
\begin{equation}\label{eq:cvc}
\vc :=  (u_c^I + u_c^{II}) - \dashint u_p.
\end{equation}

\subsection{Reynolds stresses}
Let us distribute $\pat(u_c^I + u_c^{II})$ and $\cR_0^\epsilon$ in \eqref{eq:rey1} as follows
\begin{equation}\label{eq:rey2}
\begin{aligned}
\pat u_1 + \divv (u_1\otimes u_1)   -  \divv  \A (\Do) + \nabla \pi_1&= \pat  \left(\Po \vp + u_c^{II} \right) + \divv (u_0 \otimes \vp + \vp \otimes u_0) \\
&+ \divv (u_0 \otimes \vc + \vc \otimes u_0 + \vp \otimes \vc + \vc \otimes \vp + \vc \otimes \vc) \\
&+ \pat u_c^{I} + \divv (\vp \otimes \vp - \cR_0^\epsilon) \\  
&+ \divv( \cR_0^\epsilon - \cR_0) \\
&- \divv \left(\A \big(D (u_0 + \vp + \vc) \big) - \A (D u_{0}) \right) - \nabla \pi_p  
\end{aligned}
\end{equation}
We rewrite the r.h.s.\ of \eqref{eq:rey2} further, recasting it into a divergence form.
\vskip 1mm
(i. First line of r.h.s.\ of \eqref{eq:rey2}) The definition \eqref{eq:cvp} of $\vp$, the definition \eqref{eq:def-mikado-d} of $W^k$, and $0=
k \cdot \nabla \Psi_{\lambda_2}^k (x) $ via property (i) of \eqref{sMikp}, give together
\begin{equation}\label{eq:defRlina}
\pat  \vp (t)= \sum_{k \in K} \Psi^k_{\lambda_2}k  \; \pat \left( a_k (\cdot, t) \Phi^k_{\mu, \lambda_1} (\cdot - \omega k t) \right).
\end{equation}
Using the above formula and the definition \eqref{eq:precvc} of $u_c^{II}$, we define the antidivergence of the first line of r.h.s.\ of \eqref{eq:rey2} 
%
\begin{equation}\label{eq:defRlin}
\begin{aligned}
R_{lin} &:= \sum_{k \in K}  \idivn_N \left(\pat \left( a_k (\cdot, t) \Phi^k_{\mu, \lambda_1} (\cdot - \omega k t) \right),  \Psi^k_{\lambda_2}k  \right) - \idivn^2_N \left( k\cdot \nabla \pat \left( a_k (\cdot,t)  \Phi^k_{\mu, \lambda_1} (\cdot - \omega k t \right) , \Psi^k_{\lambda_2}  \right) \\
&+ (u_0 \otimes \vp + \vp \otimes u_0).
\end{aligned}
\end{equation}
\vskip 1mm
(ii. Second line of r.h.s.\ of \eqref{eq:rey2})
\begin{equation}\label{eq:defRcorr}
R_{corr} := u_0 \otimes \vc + \vc \otimes u_0 + \vp \otimes \vc + \vc \otimes \vp + \vc \otimes \vc.
\end{equation}
\vskip 1mm
(iii. Third line of r.h.s.\ of \eqref{eq:rey2})
Here we use an important idea of \cite{BV19}. Via the definition \eqref{eq:uc1def} of $u_c^I$ and the property \eqref{eq:twcancel} we have
\[
\pat u_c^{I} =- \Po P_H  \sum_{k \in K}   a^2_k \, \divv \left( W^k  \otimes W^k \right) + 
\Po P_H \sum_{k \in K}  (\pat a^2_k) Y^k.
\]
Adding the above identity to \eqref{eq:Rp3d} that expresses $\divv( \vp \otimes \vp - \cR^\epsilon_0)$,  the $`\divv( W^k \otimes W^k)$' term cancel out and one has
\[
\begin{aligned}
\pat u_c^{I}  + \divv (\vp \otimes \vp - \cR_0^\epsilon) 
- \nabla \tilde \varrho   =&\sum_{k \in K} \left(W^k \otimes W^k - k \otimes k \right)  \nabla a^2_k  +   \Po P_H \sum_{k \in K}  (\pat a^2_k) Y^k.
\end{aligned}
\]
Let us thus define, leaving out $\nabla \tilde \varrho$ since it will be accounted for by the pressure perturbation $q_p$,
\begin{equation}\label{eq:defRq}
 R_{quadr}  :=   \sum_{k \in K} \tilde \idivn_1  \left( \nabla a_k^2,  \Po \left( W^k \otimes W^k \right) \right) + a_k^2 \Big( \dashint  W^k \otimes W^k \!- k \otimes k \Big) + \divn  \Po P_H \left( (\pat a^2_k) Y^k\right),
\end{equation}
so that $\divv R_{quadr} =  \partial_t u_c^I + \divv (u_p \otimes u_p - \cR_0^\epsilon) - \nabla \tilde \varrho.$
\vskip 1mm
(iv. Fourth line on r.h.s. \ of \eqref{eq:rey2}) Let
\begin{equation}
\label{eq:defRm}
R_{moll} := \cR_0^\epsilon - \cR_0.
\end{equation}

(v. Last line of r.h.s.\ of \eqref{eq:rey2})
Let
\begin{equation}\label{eq:defRA}
R_{\A} := - \left(\A \big(D (u_0 + \vp + \vc) \big) - \A (D u_{0}) \right).
\end{equation}
\subsection{Pressure}\label{ssec:pres}
In order to balance for  $ \nabla \tilde \varrho $ and ensure null-trace of $\cR_1$, we choose  $\pi_p$ such that
\[
\nabla \pi_p =  \nabla \tilde \varrho   + \divv\frac{1}{d} tr (R_{lin} + R_{corr} +  R_{quadr} + R_{\A}) \Id, \qquad \text{i.e.}
\]
\[
\pi_p (x,t) + c (t) := \tilde \varrho +\frac{1}{d} tr (R_{lin} + R_{corr} +  R_{quadr} + R_{\A}),
\]
having freedom in choosing $c(t)$, we set it so that $\pi_p (x,t)$ is null-mean. 
\subsection{Conclusion} 
Comparing our choices \eqref{eq:defRlin}, \eqref{eq:defRcorr}, \eqref{eq:defRq}, \eqref{eq:defRm} and \eqref{eq:defRA} for r.h.s.\ of \eqref{eq:rey2} with its l.h.s.\ we have reached
\begin{equation}\label{eq:nNR1}
\pat u_1 + \divv (u_1\otimes u_1)  - \divv  \A (\Do)+\nabla \pi_1 =  -\divv \cR_1
\end{equation}
with
\begin{equation}\label{eq:Rsplit}
R_1 :=- (R_{lin} +  R_{corr} +   R_{quadr} + R_{moll} + R_{\A}).
\end{equation}
Notice that tensor $R_1$ is symmetric, since its components are symmetric (cf.\ respective definitions and Proposition \ref{prop:invdiv}).

\section{Estimates}\label{sec:est}
We continue the proof of the main Proposition \ref{prop:main}. 
In the previous section, given $(u_0, \pi_0, R_0)$ solving the non-Newtonian-Reynolds system, we defined $u_1 = u_0 + u_p + u_c$, $\pi_1$, and the new Reynolds stress $R_1$, required by Proposition \ref{prop:main}. 
The perturbation $u_p$, the corrector $u_c$ and the error $R_1$ depend on the six parameters 
\begin{equation}
\label{eq:parameters}
 \epsilon >0, \quad \mu, \lambda_1, \lambda_2, \omega \in \N, \quad N \in \NN,
\end{equation}
which satisfy the condition \eqref{eq:general-assumpt-mikado}. The mollification parameter $\epsilon$ helps to avoid degeneracies or singularities of $\A$.  Let us immediately fix it so that
\begin{equation}
\label{eq:def-epsilon}
\epsilon \leq \frac{\delta}{2^7 d}, \qquad  |\cR_0^\epsilon(t) - \cR_0(t)|_1 \leq \frac{\eta}{2} \quad \text{ for every $t \in [0,1]$},
\end{equation}
where $\eta>0$ is the parameter appearing in the statement of the main Proposition \ref{prop:main}.

In this section we estimate $u_p, u_c, R_1$ and the energy gap of the new solution $u_1$ in terms of the remaining five parameters  $\mu, \lambda_1, \lambda_2, \omega, N$. They will be appropriately chosen in Section \ref{sec:proppf} so that \eqref{eq:mp_up} -- \eqref{eq:e_contr} hold, thus concluding the proof of the main Proposition \ref{prop:main}.

\begin{remark}[Silent assumptions]
For results of this section, we assume without writing it explicitly at each occasion: $(u_1, \pi_1, R_1)$ is the triple constructed in the previous section, the set of directions $K$ is fixed by Remark \ref{rmk:directions-fixed}, the relations \eqref{eq:general-assumpt-mikado} hold, the assumptions of Proposition \ref{prop:main} hold, the choice \eqref{eq:def-epsilon} holds.
\end{remark}

\subsection{Constants}\label{ssec:con}
We distinguish two types of constants: the uniform ones ($M$'s) and the usual ones  ($C$'s). None depend on $\mu,\lambda_1, \lambda_2, \omega$. 

More precisely, we denote by $M$ any constant depending only on the following parameters
\begin{equation}
\label{eq:uniform-obj}
\begin{aligned}
&\nu_0, \nu_1 &&  \text{the parameters entering in the definition of the non-Newtonian tensor field $\A$} \\
&q &&  \text{the exponent entering in the definition of the non-Newtonian tensor field $\A$} \\
&e  && \text{the energy profile fixed in the assumptions of Prop. \ref{prop:main}} \\
& \Phi, \Psi && \text{the profiles used in the definition of the Mikado functions in Section \ref{sec:mik}} \\
& K && \text{the fixed set of directions, cf.\ Remark \ref{rmk:directions-fixed}}.
\end{aligned}
\end{equation}
Consequently, any universal constant $M$ remains uniform over the convex integration iteration.  We will not explicitly write the dependence of $M$'s on the objects in \eqref{eq:uniform-obj}. 

\smallskip

On the other side, we will denote by $C$ (possibly with subscripts) any constant  depending not only on the universal quantities \eqref{eq:uniform-obj}, but also on 
\begin{equation}
\label{eq:obj-no-written}
(u_0, \pi_0, R_0), \delta, \eta \qquad \text{given in the assumptions of Prop. \ref{prop:main}}, 
\end{equation}
\begin{equation}
\label{eq:obj-written}
\begin{aligned}
&i,r&&  \text{if we are estimating the $W^{i,r}$ norm} \\
&N && \text{the (not yet fixed) parameter in \eqref{eq:parameters}},
\end{aligned}
\end{equation}
Constants $C$ will be controlled within each iteration step (i.e.\ in the proof of  Proposition \ref{prop:main}) by appropriate choices $\mu,\lambda_1, \lambda_2, \omega$.

\subsection{Preliminary estimates: control of $a_k$}\label{ssec:chia}
\begin{prop}\label{prop:ak}
Coefficients $a_k$ defined via \eqref{eq:def-ak} satisfy
\begin{equation}\label{eq:chipaa}
\left| a_k (t)  \right|_2 \le 2 \delta^\frac12,
\end{equation}
\begin{equation}\label{eq:chica2}
\qquad \qquad \quad  \left| a_k\right|_{C^i_{x,t}} \le C_{i} \qquad \text{ for $i \ge 0$ }
\end{equation}
\end{prop}
\begin{proof}
The definition  \eqref{eq:gnot} of $\gamma_0$, assumption \eqref{eq:ass_e}, and the assumed bounds on the energy profile $e \in  [\frac{1}{2},1]$ yield  $\gamma_0(t) \in [\frac{\delta}{8d},\delta]$.
This and the choice \eqref{eq:def-epsilon} of $\epsilon$ give, via the definition \eqref{eq:motiv_rho} of $\varrho$,
\[
\varrho(x,t) \leq 2 \epsilon + 2 |\cR_0^\epsilon(x,t)| + \gamma_0(t) \leq 3 (|\cR_0^\epsilon(x,t)| + \delta).
\]
Therefore, since $a^2_k = \varrho \Gamma^2_k$ by its definition  \eqref{eq:def-ak}, whereas $\Gamma_k \leq 1$ by Lemma \ref{lem:NasUFS},
\[
|a_k(t)|_2^2 = \int_{\Td} \varrho \Gamma^2_k 
\leq 3 (\delta + |\cR_0^\epsilon(t)|_1),
\]
The last inequality together with \eqref{eq:Rdelta-epsilon} yields \eqref{eq:chipaa}.

Using smoothness of $\Gamma_k$ supported in the compact set \eqref{eq:unifN} and that $\varrho \ge \gamma_0 \ge \frac{\delta}{8d}$, one has \eqref{eq:chica2}.
\end{proof}
\subsection{Estimates for velocity increments}\label{ssec:u} 
We will use now the improved H\"older inequality \eqref{eq:imprH} and the preliminary estimates to control $u_p, u_c$. 

\begin{prop}[Estimates for the principal increment $u_p$]\label{prop:up}
For every $r \in [1,\infty]$
\begin{equation}\label{eq:uplp<}
|\vp (t)|_r \le C_{r} \mu^{\frac{d}{2} - \frac{d}{r}}, 
\end{equation}
\begin{equation}\label{eq:upwp}
|\vp (t)|_{W^{1,r}} \le  C_{r} \lambda_2 \mu^{\frac{d}{2} - \frac{d}{r}}.
\end{equation}
Moreover, there is a universal constant $M>0$ such that
\begin{equation}\label{eq:uplp2}
|\vp (t)|_2 \le  M \delta^{1/2}  + C \lambda_1^{-\frac{1}{2}}.
\end{equation}
\end{prop}
\begin{proof} The definition \eqref{eq:cvp} of $\vp$ 
yields
\begin{equation}\label{eq:uplpf}
|\vp (t)|_r 
 \le \sum_{k \in K}   |a_k  (t) W^k (t)|_r.
 \end{equation}
Using \eqref{eq:chica2} to control $|a_k|_\infty$ and \eqref{eq:fdc} to control $|W^k (t)|_r$, we obtain  \eqref{eq:uplp<}. An analogous computation gives \eqref{eq:upwp}. 

Notice that for $r=2$ the power of $\mu$ in \eqref{eq:uplp<} is $0$: for this reason, to reach \eqref{eq:uplp2}, one needs more care. Recall from Proposition \ref{prop:cltmik} that $W^k$ is $\lambda_1$-periodic. Therefore we may apply improved H\"older inequality (Proposition \ref{lem:improved-holder}) to the r.h.s.\ of \eqref{eq:uplpf} and use \eqref{eq:fdc} for $ |W^k(t)|_2$ to obtain
\[
\begin{aligned}
|\vp (t)|_2 \le M \sum_{k \in K}   |a_k(t)|_2  +  {\lambda_1}^{-\frac{1}{2}} |a_k(t)|_{C_x^1} 
\end{aligned}
\]
Let us now use \eqref{eq:chipaa}, \eqref{eq:chica2} to control $a_k$, reaching \eqref{eq:uplp2}.
\end{proof}

Now we deal with the corrector $u_c$. In order to shorten the related formulas, let us introduce
\begin{equation}\label{eq:L}
L(r) :=  \frac{ \mu^{d-\frac{d}{r}}}{\omega}  + \frac{\lambda_1 \mu}{\lambda_2}  \mu^{\frac{d}{2} - \frac{d}{r} } \left[ 1 + \lambda_2^2 \left( \frac{\lambda_1 \mu}{\lambda_2} \right)^{N} \right]
\end{equation}
Observe that $r \mapsto L(r)$ is a non-decreasing map. 
\begin{prop}[Estimates for the corrector $u_c$]\label{prop:uc}
For every $r \in (1, \infty)$ it holds
\begin{equation}\label{eq:ucwp}
|\vc (t)|_{\dot W^{i,r}} \le C_{i,d,r,N} \lambda^i_2 L(r).
\end{equation}
\end{prop}
\begin{proof}
Recall that $\vc =  (u_c^I + u_c^{II}) - \dashint_\Td u_p$ by its definition \eqref{eq:cvc}, with $\vc^I$ defined by \eqref{eq:uc1def} and $\vc^{II}$ defined by \eqref{eq:precvc}. 


The Calder\'on-Zygmund estimate, 
 \eqref{eq:chica2}, and \eqref{eq:fdc} give
\begin{equation}
\label{eq:eucI}
|\vc^I (t)|_{\dot W^{i,r}} \le  
C_{r} \frac{\lambda^i_2 \mu^{d - \frac{d}{r}}}{\omega}
\end{equation}
For the estimate of $u_c^{II}$ in $\dot W^{i,r}$ we use the inequality \eqref{eq:inv_impr2div} for $j=i+1$, $s=\infty$, which yields
\[
| u_c^{II}|_{\dot W^{i,r}}   \le C_{j,d,r,N} \lambda_2^{i+1} \Big( \frac{1}{\lambda_2^{2}} | \nabla( a_k \Phi^k_{\mu, \lambda_1})|_{r} + \frac{1}{\lambda_2^{N}} |\nabla^{N+1}( a_k \Phi^k_{\mu, \lambda_1})|_{r}  + \frac{1}{\lambda_2^{2N+i+1}} |\nabla^{2N +i+2}( a_k \Phi^k_{\mu, \lambda_1})|_{r} \Big).
\]
Now we estimate $a_k$ by \eqref{eq:chica2} and $\Psi^k_{\mu, \lambda_1}$ by \eqref{eq:cf1}. Using the assumption \eqref{eq:general-assumpt-mikado}, we arrive at
\begin{equation}\label{eq:uc2estW}
\begin{aligned}
| u_c^{II}|_{\dot W^{i,r}} \le   C_{i,d,r,N} {\lambda^i_2} \frac{\lambda_1 \mu}{{\lambda_2}}  \mu^{\frac{d}{2} -\frac{d}{r}}  \left[ 1+  \lambda_2^2     \left(\frac{\lambda_1 \mu}{{\lambda_2}} \right)^{N} \right].
\end{aligned}
\end{equation}
Recall $\vc =  (u_c^I + u_c^{II}) - \dashint_\Td u_p$ by its definition \eqref{eq:cvc}.
Therefore, putting together \eqref{eq:eucI}, \eqref{eq:uc2estW}, and
\[
\Big| \dashint u_p \Big| \le \frac{\lambda_1 \mu}{{\lambda_2}}  \mu^{-\frac{d}{2}} \le \frac{\lambda_1 \mu}{\lambda_2}  \mu^{\frac{d}{2} - \frac{d}{r} } 
\]
valid via \eqref{eq:mosc} with $r=1$ and $v = \Psi^k$, 
yields \eqref{eq:ucwp}.
\end{proof}

\subsection{Estimates on the Reynolds stress}\label{ssec:R}
Recall \eqref{eq:Rsplit}
\[
R_1 :=- (R_{lin} +  R_{corr} +   R_{quadr} + R_{moll} + R_{\A}). 
\]
In this section we estimate each term of $R_1$. For our further purposes $L^1$ estimates suffice, but due to using Calder\'on-Zygmund theory, some estimates are phrased as $L^r$ ones.

\begin{prop}[Estimates on the principal Reynolds $R_{quadr}$]\label{prop:Rquadr}
For every $r \in (1, \infty)$, it holds
\begin{equation}\label{eq:Rq}
|R_{quadr}(t)|_r \le  C_{r}  \left( (\omega^{-1}  + \lambda^{-1}_1 ) \mu^{d-\frac{d}{r}} +  \lambda_2^{-1} \lambda_1 \mu \right) 
\end{equation}
\end{prop}
\begin{proof}
Recall the definition \eqref{eq:defRq} of $R_{quadr}$. Let us estimate its three terms in order of their appearance. 

(i) The first term of $R_{quadr}$ is the sum over $k$ of $\tilde \idivn_1 (\nabla a_k^2, \Po (W^k\otimes W^k))$. The term $\Po (W^k\otimes W^k))$ is null mean and it oscillates at the frequency $\lambda_1$, since $W^k$ does. Therefore \eqref{eq:inv_imprA} with \eqref{eq:chica2} and \eqref{eq:fdc} give
\begin{equation}
\label{eq:Rqa}
|\tilde \idivn_1 (\nabla a_k^2, \Po (W^k \otimes W^k))|_r
 \le C_{r} \lambda^{-1}_1 |\Po (W^k \otimes W^k))|_r |\nabla^2 (a_k^2)|_\infty  \le   C_{r}  \lambda^{-1}_1\mu^{ d - \frac{d}{r}}.
\end{equation}

(ii) The second term of $R_{quadr}$ is the sum over $k$ of $a_k^2 ( \dashint W^k\otimes W^k - k \otimes k)$. We use estimate \eqref{eq:chica2} to control $a_k$ terms and \eqref{eq:mean-value-mikado} to control the Mikado terms
\begin{equation}
\label{eq:Rqb}
 \Big| a_k^2 \Big( \dashint W^k\otimes W^k- k \otimes k \Big)  \Big|_r  \le |a^2_k|_r  \Big| \dashint W^k \otimes W^k - k \otimes k  \Big| \leq C_r \frac{\lambda_1 \mu}{\lambda_2}.
\end{equation}

\smallskip

(iii) The last term is the sum over $k$ of $\divn P_H \left( \Po (\pat a^2_k) Y^k \right)$. We deal with $\divn$ via \eqref{eq:inv_stnd} and with $P_H$ via Calder\'on-Zygmund, control  $\pat a_k$ using \eqref{eq:chica2}, and use \eqref{eq:fdc} to estimate $Y^k$. Hence
\begin{equation}
\label{eq:Rqc}
\left|\divn P_H \left( \Po (\pat a^2_k) Y^k \right)\right|_r  (t) \le C_r \left|   \Po (\pat a^2_k) Y^k\right|_r  (t) \le C_{r} |Y^k(t)|_r \leq C_{r}  \frac{ \mu^{d-\frac{d}{r}}}{\omega}.
\end{equation}
Together,  \eqref{eq:Rqa},  \eqref{eq:Rqb},  \eqref{eq:Rqc} yield \eqref{eq:Rq}. 
\end{proof}

\begin{prop}[Estimate on $R_{lin}$] 
It holds 
\begin{equation}\label{eq:Rlinest}
|R_{lin}|_1 (t) \le C_{N} \bigg[ \mu^{-\frac{d}{2}}  + 
\frac{  \lambda_1  \mu^{1-\frac{d}{2}} \omega  }{\lambda_2} \Big( 1 + \lambda_2^2 \Big(\frac{\lambda_1 \mu}{\lambda_2} \Big)^{N} \Big)  \bigg] .
\end{equation}
\end{prop}

\begin{proof}
Recall the definition \eqref{eq:defRlin} of $R_{lin}$. It involves three terms, which we estimate in order of their appearance. 

(i) The first term of $R_{lin}$ is the sum over $k$ of \[\idivn_N \big( (\pat a_k) \Phi^k_{\mu, \lambda_1}  + \omega a_k  k \cdot \nabla \Phi^k_{\mu,  \lambda_1}, \Psi^k_{\lambda_2} k \big).\]
Using \eqref{eq:inv_imprA} with  $|u|_s = |\Psi^k|_\infty$, the assumption \eqref{eq:general-assumpt-mikado},  and disposing of $a_k$ as usual, one has
\begin{equation}
\label{eq:Rlinb}
| \idivn_N \left(\pat \left( a_k (\cdot, t) \Phi^k_{\mu, \lambda_1} (\cdot - \omega k t) \right),  \Psi^k_{\lambda_2}k  \right)  |_1 (t)  \le C_{d, N} \frac{\omega \lambda_1 \mu^{1 - d/2}}{\lambda_2} \Big[ 1 +  \lambda_2 \Big( \frac{\lambda_1 \mu}{\lambda_2} \Big)^N \Big].
\end{equation}

(ii) The second term of $R_{lin}$ is the sum over $k$ of $$\idivn^2_N (k \cdot \nabla \pat ( a_k \Phi^k_{\mu, \lambda_1} (\, \cdot \, - \omega k t ), \Psi^k_{\lambda_2}).$$ 
We observe that 
\begin{equation}\label{eq:afl2}
|\nabla^{i+1}  \pat \left( a_k \Phi^k_{\mu, \lambda_1} (\, \cdot \, - \omega k t  ) \right)|_1
 \leq |\partial_t a_k \Phi^k_{\mu, \lambda_1}|_{W^{i+1,1}} + \omega |k| | a_k \nabla \Phi^k_{\mu, \lambda_1}|_{W^{i+1,1}} \\
 \leq C \omega \lambda_1^{i+2} \mu^{i+2 - \frac{d}{2}}.
\end{equation}
Using the computation \eqref{eq:afl2} in \eqref{eq:inv_impr2div} with $j=0$ and $|u|_s = |\Psi^k|_\infty$, we get
\begin{equation}\label{eq:Rlinc}
\big| \idivn^2_N \left(k \cdot \nabla \pat \left( a_k \Phi^k_{\mu, \lambda_1} (\, \cdot \, - \omega k t )  \right), \Psi^k_{\lambda_2}  \right) \big|_1 \leq C_{d,N}  \, \omega \mu^{-\frac{d}{2}}   \Big( \frac{\lambda_1 \mu}{\lambda_2} \Big)^2    \left[ 1   + \lambda_2^2 \Big( \frac{\lambda_1 \mu}{\lambda_2} \Big)^N    \right].
\end{equation}

(iii) The third term of $R_{lin}$ equals $u_0 \otimes \vp + \vp \otimes u_0$, so we write using \eqref{eq:uplp<}
\begin{equation}\label{eq:Rlind}
 |u_0|_\infty |\vp|_1 \le C\mu^{-\frac{d}{2}}.
\end{equation}

Putting together \eqref{eq:Rlinb},  \eqref{eq:Rlinc}, \eqref{eq:Rlind}, and observing that the right-hand sides of both \eqref{eq:Rlinb} and \eqref{eq:Rlinc} are estimated by  $C_{d,N}  \, \omega \mu^{-\frac{d}{2}}   \frac{\lambda_1 \mu}{\lambda_2}    [ 1   + \lambda_2^2 ( \frac{\lambda_1 \mu}{\lambda_2})^N]$ thanks to \eqref{eq:general-assumpt-mikado}, one has \eqref{eq:Rlinest}.
\end{proof}


\begin{prop}[Estimates on $R_{corr}$]\label{prop:Rc} Let  $L(2)$ be given by \eqref{eq:L}. It holds
\begin{equation}\label{eq:Rc1}
\begin{aligned}
|R_{corr} (t)|_1 \le C_{N} \big(L(2) + L^2(2) \big).
\end{aligned}
\end{equation}
\end{prop}
\begin{proof} 
By  the definition  \eqref{eq:defRcorr} of $R_{corr}$ we have
\[
|R_{corr}|_1 (t) \le C(|u_0|_2  |\vc|_2 + |\vp|_{2} |\vc|_{2} + |\vc|^2_{2})(t),
\]
since $|\vp|_{2} \le C$ via \eqref{eq:uplp<} and $|\vc|_{2} \le  C_N L(2)$ via \eqref{eq:ucwp}, we have \eqref{eq:Rc1}. 
\end{proof}

\begin{prop}[Estimates on the dissipative Reynolds $R_{\A}$]\label{prop:RA}
For $q \in (1, \infty)$ being the growth parameter of $\A$ it holds 
\begin{equation}\label{eq:RA}
|R_{\A}(t)|_r \le \left\{ 
\begin{aligned}
 & C   \big(  {( \lambda_2 \mu^{\frac{d}{2} - \frac{d}{r}})^{q-1}} + (\lambda_2 L(r))^{q-1}  \big) &\text{ for } \nu_0=0, q { \leq } 2, \;&&\text{any } r>1, \\
 & C  \big(  { \lambda_2 \mu^{\frac{d}{2} - \frac{d}{r}}}  +  \lambda_2 L(r) \big) &\text{ for } \nu_0>0, q {\leq} 2, \;&&\text{any } r>1, \\
 & C \Big( { \lambda_2 \mu^{\frac{d}{2} - \frac{d}{r(q-1)}}} +   { \big( \lambda_2 \mu^{\frac{d}{2} - \frac{d}{r(q-1)}} \big)^{q-1}} \\
 & \qquad  +  \lambda_2 L(r(q-1)) +  (\lambda_2 L(r(q-1)))^{q-1} \Big)&\text{ for } q\ge2, \;&&\text{any } r(q-1)>1.
\end{aligned}
\right.
\end{equation}
\end{prop}
\begin{proof}
By definition \eqref{eq:defRA}, we have \[|R_{\A}| = |\A(Du_0 + Du_p + Du_c) - \A(Du_0)|.\] Therefore the inequality \eqref{eq:a-est-1} gives the pointwise estimate
\[
|R_{\A}| \le \left\{ 
\begin{aligned}
 & C_{\nu_1} |Du_p + Du_c|^{q-1} &\text{ for } \nu_0=0, q\leq 2 \\
 & C_{\nu_0} |Du_p + Du_c| &\text{ for } \nu_0>0, q\leq2 \\
 & C_{q, \nu_0, \nu_1}  |Du_p + Du_c| \left(1 + |Du_0|^{q-2} + |Du_1|^{q-2} \right) &\text{ for } q\ge2
\end{aligned}
\right.
\]
Using Jensen inequality and $\frac{1}{q-1} {\geq} 1$ in the first case, and H\"older inequality with $q-1$, $\frac{q-1}{q-2}$ in the last case, one has
\[
|R_\A|_r \le \left\{ 
\begin{aligned}
 & C |D u_p + Du_c|_r^{q-1} &\text{ for } \nu_0=0, q {\leq} 2 \\
 & C |Du_p + Du_c|_r &\text{ for } \nu_0>0, q {\leq} 2 \\
 & C |Du_p + Du_c|_{r(q-1)} \left(1 + |Du_0|_{r(q-1)}^{q-2} + |Du_1|_{r(q-1)}^{q-2} \right) &\text{ for } q\ge2
\end{aligned}
\right.
\]
%
%
For any $s \in (1, \infty)$ the estimate \eqref{eq:ucwp} controls  $|D u_c|_s$ via $\lambda_2 L(s)$, {whereas $\lambda_2 \mu^{d/2 - d/s}$ controls $|D u_p|_s$ thanks to \eqref{eq:upwp}}. This closes the case $q<2$ of \eqref{eq:RA}. Recalling that $u_1 = u_0 +u_p + u_c$ and that $C$ may contain norms of $u_0$, we obtain the case $q\ge2$.
\end{proof}
\begin{remark}
For the current purpose of proving Proposition \ref{prop:main} and thus Theorem \ref{thm:main}, the case $q \leq 2$ of \eqref{eq:RA} suffices. We included already the case $q\ge2$, because it is needed to prove Theorem \ref{thm:two}.
\end{remark}

Immediately from the definition of $R_{moll}$ in \eqref{eq:defRm} and the choice of $\epsilon$ in \eqref{eq:def-epsilon} we have
\begin{prop}[Estimate on $R_{moll}$] 
It holds
\begin{equation}
\label{eq:Rm}
|R_{moll}(t)|_1 \leq \eta/2 
\end{equation}
where $\eta$ is the parameter appearing in the assumptions of Proposition \ref{prop:main}. 
\end{prop}

\subsection{Estimates on the energy increment}\label{ssec:ener}
We intend to approach the desired energy profile $e(t)$, i.e.\ perform the step \eqref{eq:ass_e} $\to$  \eqref{eq:e_contr}. Let us thence define $\delta E$ as follows
\begin{equation}\label{eq:Edef}
\delta E (t):= \Big|e(t) \Big(1 - \frac{\delta}{2} \Big) - \Big( \int |u_1|^2 (t) + 2 \int_0^t \int  \A (D u_1) D u_1 \Big)  \Big|.
\end{equation}
Recall quantities $L$ of \eqref{eq:L}. We will show
\begin{prop}[energy iterate]\label{prop:ei} For $q \in (1, \infty)$ being the growth parameter of $\A$ it holds
\begin{equation}\label{eq:E}
\delta E(t)  \le  \frac{\delta}{16} e(t)  + C_{N} \left(\lambda^{-1}_1 + {\mu^{-\frac{d}{2}}} + L(2) + L(2)^2 + { \lambda_2 \mu^{\frac{d}{2} - \frac{d}{q}} + \big( \lambda_2 \mu^{\frac{d}{2} - \frac{d}{q}} \big)^q  } + \lambda_2 L(q) + \big( \lambda_2 L(q) \big)^q  \right).
\end{equation}
\end{prop}
\begin{proof}
Recall \eqref{eq:Rp3}. Taking its trace and recalling that $\cR_0^\epsilon $ is traceless we have
\begin{equation}\label{eq:el1}
  |\vp|^2  =  d \varrho 
  +\sum_{k \in K} a^2_k \Po |W^k|^2+ a_k^2 \Big( \dashint |W^k|^2  - |k|^2\Big).
\end{equation}
By the definition \eqref{eq:motiv_rho} it holds
$d\varrho = 2 d \sqrt{\epsilon^2 + |\cR_0^\epsilon|^2} + d\gamma_0$, therefore
\[
|\vp|^2  -d \gamma_0 =  2d\sqrt{\epsilon^2 + |\cR_0^\epsilon|^2}+\sum_{k \in K} a^2_k \Po |W^k|^2 + a_k^2 \Big( \dashint |W^k|^2 - |k|^2 \Big)
\]
Integrating and using  $\sqrt{\epsilon^2 +|x|^2} \le \epsilon +|x|$, 
we have 
\begin{equation}\label{eq:el1ip}
\Big| \int   |\vp|^2  - d \gamma_0 \Big|
\le 2d \epsilon + 2d |\cR^\epsilon_0(t)|_1   +\sum_{k \in K} \Big| \int a^2_k \Po |W^k|^2 \Big|  +  |a_k|_2^2 \Big| \dashint |W^k|^2 - |k|^2  \Big|
\end{equation}
We estimate the first two terms of the r.h.s.\ of \eqref{eq:el1ip} using  \eqref{eq:def-epsilon} and \eqref{eq:Rdelta-epsilon} as follows
\[
2d \epsilon + 2d |\cR^\epsilon_0(t)|_1  \le  \frac{\delta}{2^6} + \frac{\delta}{2^6} \le  \frac{\delta}{2^4} e(t),\]
where in the second inequality we used the assumption $e(t) \ge \frac12$. This in \eqref{eq:el1ip} yields

\begin{equation}\label{eq:el1i}
\Big| \int |\vp|^2  - d \gamma_0 \Big|
\le \frac{\delta}{2^4}e(t)  +\sum_{k \in K} \Big| \int a^2_k \Po |W^k|^2 \Big|  + |a_k|_2^2 \Big| \dashint |W^k|^2 - |k|^2  \Big| \end{equation}
The first integral of r.h.s.\ of \eqref{eq:el1i} involves a $\lambda_1$-oscillating function $\Po |W^k|^2$, recall Proposition \ref{prop:cltmik}. Therefore, using \eqref{eq:mosc}, then \eqref{eq:chica2} to control $a_k$ and \eqref{eq:fdc} for $W^k$, we have
\begin{equation}
\label{eq:el1ia}
\Big| \int a^2_k \Po |W^k|^2 \Big| \le C   \lambda_1^{-1} |a^2_k|_\infty |(\Po |W^k|)|^2_2 \le  C   \lambda_1^{-1}.
\end{equation}
For the integral following the second sum in \eqref{eq:el1i}, we use \eqref{eq:mean-value-mikado} and \eqref{eq:chica2} to get
\begin{equation}
\label{eq:el1ib}
 |a_k|_2^2 \Big| \dashint |W^k|^2 - |k|^2  \Big|  \leq C \frac{\lambda_1 \mu}{\lambda_2}. 
\end{equation}
We plug \eqref{eq:el1ia} and \eqref{eq:el1ib} to  \eqref{eq:el1i} and obtain
\begin{equation}
\label{eq:el3}
\Big| \int  |\vp|^2  - d \gamma_0 \Big| \le \frac{\delta}{2^4} e(t) + C \Big( \frac{1}{\lambda_1}   + \frac{\lambda_1 \mu}{\lambda_2} \Big).
\end{equation}

Use $u_1 = u_0 + u_p + u_c$ in the definition \eqref{eq:Edef} of $E$ to write for the time instant $t$
\[
\delta E(t) \le \Big| \int |u_p|^2 - d \gamma_0  \Big|+ \Big| \int  |u_c|^2 + 2 (u_0 u_c + u_0 u_p + u_p u_c)  \Big| + 2  \Big|\int_0^t \int  \A (D u_1) D u_1 - \A (D u_0) D u_0 \Big|,
\]
where a cancellation occurs, thanks to the definition \eqref{eq:gnot} of  $\gamma_0$. Inequality \eqref{eq:el3} allows to control the first term of the r.h.s.\ above. For the last term we use  \eqref{eq:a-est-3}, next H\"older inequality with $q$, $\frac{q}{q-1}$, and finally $u_1 = u_0 +u_p + u_c$ to get
\begin{equation*}
\int | \A (D u_1) D u_1 - \A (D u_0) D u_0| 
\leq C \left(   |Du_p + Du_c|_q + |Du_p + Du_c|_q^q \right).
\end{equation*}
Thus, integrating in time over $[0,t] \subseteq [0,1]$,
\begin{equation*}
\int_0^t \int   |\A (D u_1) D u_1 - \A (D u_0) D u_0| \leq C \sup_{\tau \in [0,1]}  \left(   |Du_p(\tau) + Du_c(\tau)|_q + |Du_p(\tau) + Du_c(\tau)|_q^q \right).
\end{equation*}
Consequently
\[
\begin{aligned}
\delta E(t)  & \le \frac{\delta}{2^4} e(t) + C \Big( \frac{1}{\lambda_1}   + \frac{\lambda_1 \mu}{\lambda_2}  \Big)\\
&+ C \left(|u_c|_2^2  + |u_0|_2 |u_c|_2 +  |u_0|_\infty |u_p|_1 +  |u_c|_{2} |u_p|_2 \right) (t) \\
& + C \sup_{\tau \in [0,1]}  \left(   |Du_p(\tau) + Du_c(\tau)|_q + |Du_p(\tau) + Du_c(\tau)|_q^q \right).
\end{aligned}
\] 
The terms in the second line above are estimated, using \eqref{eq:ucwp} for $u_c$ and \eqref{eq:uplp<} for $u_p$,  by $C_{N} ( L(2)^2 + L(2) + \mu^{-d/2})$. Observe that  $L(2)$ of this term can absorb  $\lambda_2^{-1} \lambda_1 \mu$ of the first line. The terms in the last line are estimated by ${ \lambda_2 \mu^{\frac{d}{2} - \frac{d}{q}} + \big( \lambda_2 \mu^{\frac{d}{2} - \frac{d}{q}} \big)^q  } + \lambda_2 L(q) + (\lambda_2 L(q))^q$, using \eqref{eq:upwp} for $Du_p$, \eqref{eq:ucwp} for $Du_c$.  We thus arrived at \eqref{eq:E}.
\end{proof}


\section{Proof of the main Proposition \ref{prop:main}}\label{sec:proppf}
Having at hand the estimates of the previous section, we are ready to show that $(u_1, q_1, R_1)$ constructed in Section \ref{sec:step} satisfy the inequalities \eqref{eq:mp_up} -- \eqref{eq:e_contr}. 

The estimates of the previous section have at their right-hand sides two type of terms: ones where the parameters $\lambda_2, \lambda_1$, $\mu$, $\omega$ are intertwined, and the remaining 
ones. These remaining ones can be made small simply by choosing the relevant parameters large. The terms with $\lambda_2, \lambda_1$, $\mu$, $\omega$ interrelated need more care, so let us focus on them. They contain two little technical nuisances: (i) appearance of $N$ and (ii) estimates for some parts of $R$ not holding in $L^1$. Let us ignore these nuisances for a moment, which is easily acceptable after recalling (i) $\idivn_\infty$ of Remark \ref{rem:ia} (which heuristically cancels the terms involving $N$) and that (ii) estimates for $R$ hold in $L^{r}$ for any $r>1$, whereas an $\epsilon$ of room is assured by the assumed sharp inequality $q<\frac{2d}{d+2}$. So for a moment let us consider estimates of Section \ref{sec:est} allowing for $|R|_1$ and disregarding the terms with $N$. After inspection, we see that smallness of their right-hand sides where  $\lambda_2, \lambda_1$, $\mu$, $\omega$ are intertwined, needed for Proposition \ref{prop:main} is precisely the smallness of \eqref{eq:comparison-mikado}, Remark \ref{rmk:comparison-mikado}. 
Therefore we will proceed as follows. 

Firstly, guided by \eqref{eq:comparison-mikado}, we will choose relation between magnitudes of $\lambda_2, \lambda_1$, $\mu$, $\omega$. To this end we postulate 
\begin{equation}\label{eq:mag}
\lambda_1 := \lambda, \quad \mu := \lambda^a, \quad \omega := \lambda^b, \quad \lambda_2 := \lambda^c,
\end{equation}
and choosing relation between magnitudes means picking $a,b,c$ so that \eqref{eq:comparison-mikado} are strictly decreasing in $\lambda$.

Secondly, we will need to make sure that when $r>1$ and $N$ appear, the relations between magnitudes do not change. This will be achieved by choosing $N$ large and $r$ small in relation to $a,b,c$.

Finally, we will send $\lambda \to \infty$ to reach \eqref{eq:mp_up} -- \eqref{eq:e_contr}.

\subsection{Picking magnitudes $a,b,c$}\label{ssec:mag}
The requirement  that powers in \eqref{eq:comparison-mikado} rewritten in terms of 
\eqref{eq:mag} are negative reads 
\begin{equation}\label{eq:conp}
\begin{aligned}
\frac{\lambda_1 \mu}{\lambda_2} & = \lambda^{1+a-c} &&  i.e.  & 1+a-c & < 0, \\
\frac{\lambda_1 \mu}{\lambda_2} \, \cdot \, \frac{\omega}{\mu^{d/2}}& = \lambda^{1+(1 - \frac{d}{2})a-c + b} && i.e. & 1+ \Big(1 - \frac{d}{2}\Big)a-c + b & < 0, \\
\frac{\mu^{d/2}}{\omega} & = \lambda^{ad/2-b} &&  i.e.  & \frac{d}{2}a -b & < 0, \\
\lambda_2 \mu^{\frac{d}{2} - \frac{d}{q}} & = \lambda^{c + (\frac{d}{2} - \frac{d}{q})a} &&  i.e.  & c - \Big( \frac{d}{q} -\frac{d}{2}  \Big)a & < 0.
\end{aligned}
\end{equation}
These conditions on $a,b,c$ can be simultaneously  achieved as follows. 
\begin{enumerate}
\item The conditions not involving $b$ amount to the requirement 
\begin{equation}\label{eq:choice-c}
1+ a < c< \Big( \frac{d}{q} -\frac{d}{2} \Big)a.
\end{equation}
From the assumption $q < \frac{2d}{d+2}$ of Proposition \ref{prop:main} it follows that $d/q - d/2 > 1$. Therefore satisfying \eqref{eq:choice-c} is possible with $a$ large. More precisely, let us  pick 
\begin{equation}
\label{eq:choice-a}
a > \frac{3}{d \big( \frac{1}{q} - \frac{d+2}{2d} \big)}.
\end{equation}
Then between $1 + a$ and $(d/q - d/2)a$ there are at least two natural numbers. We then fix $c \in \N$ as the largest natural number satisfying \eqref{eq:choice-c}.
Notice that there is still at least one natural number between $1+a$ and $c$.
\item Let us fix $b \in \N$ so that
\begin{equation}
\label{eq:choice-b}
\frac{d}{2}a < b< \Big( \frac{d}{2} - 1 \Big) a + c -1.
\end{equation}
This is possible, because, as observed in point (1), there is at least one natural number between $1+a$ and $c$ and thus also between $ad/2$ and $(d/2-1)a + c -1$. The condition \eqref{eq:choice-b} automatically verifies the two conditions concerning $b$. 
\end{enumerate}

Let us denote by  $-\zeta<0$ the largest power of those appearing in \eqref{eq:conp}. We have just showed
\begin{equation}\label{eq:uei}
 \frac{\mu^{d/2}}{\omega} \le \lambda^{-\zeta}, \quad \lambda_2 \frac{\mu^{d-\frac{d}{q}}}{\omega} \le \lambda^{-\zeta}, \quad \frac{\lambda_1 \mu}{{\lambda_2}} \le \lambda^{-\zeta},   \quad  {\lambda_2}  \mu^{\frac{d}{2} - \frac{d}{q}} \le \lambda^{-\zeta},  \quad  \frac{\lambda_1 \mu}{{\lambda_2}} \cdot \frac{\omega}{\mu^{d/2}} \le \lambda^{-\zeta}.
\end{equation}

\subsection{Fixing $N$ and $r_0>1$}\label{ssec:nr}
Let us fix $N \in \N$ so that
\begin{equation}
\label{eq:choice-N}
c - \zeta N \le 0.
\end{equation}
This choice yields
\begin{equation}
\label{eq:tail-N}
1 + \lambda_2^2 \Big( \frac{\lambda_1 \mu}{\lambda_2} \Big)^N \le  1 + \lambda^{2c - 2 \zeta N} \leq 2.
\end{equation}
Using the definition \eqref{eq:L} of $L$ with \eqref{eq:uei} and \eqref{eq:tail-N}, one has
\begin{equation}
\label{eq:l-tildel-neg}
L(q) \leq C\lambda^{-\zeta}, \quad L(2) \leq C\lambda^{-\zeta}, \quad \lambda_2 L(q) \leq C\lambda^{-\zeta}.
\end{equation}
Importantly, fixing the gauge $N$ freezes all $C_N$'s in estimates to $C$. 

Let us fix also an exponent $r_0 \in (1, \infty)$ (close to $1$) such that
\begin{equation}
\label{eq:choice-r}
\begin{aligned}
\Big( d - \frac{d}{r_0} \Big) a \le \frac12.
\end{aligned}
\end{equation}
This is possible because the l.h.s. above vanishes as $r_0 \to 1$. 
%

\subsection{Obtaining \eqref{eq:mp_up} -- \eqref{eq:e_contr}}\label{ssec:pineq}
Recall that $\delta, \eta$ are given small numbers. Since $u_1-u_0 = \vp + \vc$, we  have by \eqref{eq:uplp2} and \eqref{eq:ucwp}
\begin{equation}\label{eq:finlp}
|(u_1-u_0)(t)|_{L^2} \leq M \delta^\frac12 + C \big( \lambda_1^{-1/2} + L(2) \big)  \leq M \delta^\frac12 + C \big( \lambda^{-1/2} + \lambda^{-\zeta} \big), 
\end{equation}
recalling for the latter inequality that $ \lambda_1 =  \lambda$ via \eqref{eq:mag}, and \eqref{eq:l-tildel-neg}.
Choose $\lambda \in \N$ large in relation to $C$ we thus have
\[
|(u_1-u_0)(t)|_{L^2} \leq 2M \delta^\frac12 = M_0 \delta^\frac12,
\]
defining $M_0 := \frac{M}{2}$, hence \eqref{eq:mp_up}. Notice that $M_0$ depends only on the universal constant $M$. Thus $M_0$ itself is universal, i.e.\ it may depend on the quantities \eqref{eq:uniform-obj}, but not on the  quantities \eqref{eq:obj-no-written}.

Similarly to obtaining \eqref{eq:finlp}, using \eqref{eq:upwp},  \eqref{eq:L} and \eqref{eq:ucwp} we have
\begin{equation}\label{eq:finwp}
|u_1 - u_0|_{W^{1,q}} 
 \leq C \big(   \lambda_2 \mu^{d/2 - d/q} + L(q) +  \lambda_2 L(q)  \big) \le C  \lambda^{-\zeta}, 
\end{equation}
where for the term $\lambda_2 \mu^{d/2 - d/q}$ we used \eqref{eq:uei}. Estimate \eqref{eq:mp_uwp} follows by choosing $\lambda$ big enough. 
%

Recall that $R_1 = - (R_{lin} +  R_{corr} +   R_{quadr} + R_{moll} + R_{\A})$ by its definition \eqref{eq:Rsplit}. By \eqref{eq:Rlinest} we have, with $C_N$ now fixed to $C$ by the choice \eqref{eq:choice-N} of $N$
\begin{equation}
\label{eq:rlin-exp}
|R_{lin}|_1 (t) \le C  \bigg[ \mu^{-\frac{d}{2}}  + 
\frac{  \lambda_1  \mu^{1-\frac{d}{2}} \omega  }{\lambda_2} \Big( 1 + \lambda_2^2 \Big(\frac{\lambda_1 \mu}{\lambda_2} \Big)^{N} \Big)  \bigg] \leq C(\lambda^{-ad/2} + \lambda^{-\zeta} 2),
\end{equation}
where for the second inequality we invoked $\mu = \lambda^a$ by \eqref{eq:mag}, \eqref{eq:uei}, and  \eqref{eq:tail-N}.

Similarly, using \eqref{eq:Rc1} and \eqref{eq:l-tildel-neg}
\begin{equation}
\label{eq:rc-exp}
|R_{corr} (t)|_1 \le  C \lambda^{-\zeta}.
\end{equation}

For the $L^1$-estimate of $R_{quadr}$ we need to switch to the $L^{r_0}$ estimate, where $r_0$ was fixed in \eqref{eq:choice-r}. We have, using \eqref{eq:Rq}
\begin{equation*}
|R_{quadr}(t)|_1 \leq |R_{quadr}(t)|_{r_0}  \le C ( \omega^{-1} + \lambda^{-1}_1)  \mu^{d-\frac{d}{r_0}}  + C \lambda^{-1} \lambda_1 \mu \leq C( \lambda^{(d-d/r_0)a - b} +\lambda^{(d-d/r_0)a - 1} ) + C \lambda^{-\zeta}.
\end{equation*}
Thanks to the choice of $r_0$ in \eqref{eq:choice-r}, we hence have
\begin{equation}
\label{eq:rq-exp}
|R_{quadr}(t)|_1 \leq C\lambda^{-\frac12}  + C \lambda^{-\zeta}.
\end{equation}
Similarly, for the estimate of $R_\A$ we use \eqref{eq:RA} with  $q \in (1,2)$, obtaining 
\[
\label{eq:R1im}
|R_\A(t)|_1 \le |R_{\A}(t)|_q \le \left\{ 
\begin{aligned}
 & C  \Big(  { \big(\lambda_2 \mu^{\frac{d}{2} - \frac{d}{q}} \big)^{q-1}} +  \lambda_2 L(q) \Big)^{q-1} &\text{ for } \nu_0=0, q<2,\\
 & C \big(  { \lambda_2 \mu^{\frac{d}{2} - \frac{d}{q}} } +  \lambda_2 L(q) \big) &\text{ for } \nu_0>0, q<2.
\end{aligned}
\right.
\]
Therefore by {\eqref{eq:uei}}, \eqref{eq:l-tildel-neg} and $q-1 \in (0,1)$
\begin{equation}
\label{eq:ra-exp}
|R_\A(t)|_1 \le  C \lambda^{-\zeta(q-1)}.
\end{equation}

Together, the terms $R_{lin}, R_{corr}, R_{quadr}, R_\A$ are bounded in view of, respectively, \eqref{eq:rlin-exp},  \eqref{eq:rc-exp},  \eqref{eq:rq-exp}, and \eqref{eq:ra-exp} by $C \lambda^{-\zeta'}$ with certain $\zeta'>0$:
\begin{equation*}
|R_{lin}(t)|_1 + |R_{corr}(t)|_1 + |R_{quadr}(t)|_1 + |R_\A(t)|_1 \leq  C \lambda^{-\zeta'}
\end{equation*}
Therefore, using for the remaining $R_{moll}$ the estimate \eqref{eq:Rm}, we have
\begin{equation}\label{eq:p1fr}
|R_1(t)|_1 \leq  \frac{\eta}{2} + C \lambda^{-\zeta'}.
\end{equation}
thus showing \eqref{eq:mp_R} by taking  $\lambda$ large.

Let us show the last remaining inequality \eqref{eq:e_contr}.  By \eqref{eq:E}, with $C_N$ fixed to $C$ by the choice \eqref{eq:choice-N} of $N$, we have in view of {\eqref{eq:uei} and}  \eqref{eq:l-tildel-neg}
\begin{equation}
\label{eq:e-delta-8}
\delta E(t) \leq \frac{\delta}{16} e(t) + C \left( \lambda^{-1} + {\lambda^{-\frac{ad}{2}}} +  \lambda^{-\zeta} + \lambda^{-2\zeta} + \lambda^{-\zeta} + \lambda^{-q\zeta} \right) \leq \frac{\delta}{16} e(t) + \frac{\delta}{32} \le \frac{\delta}{8} e(t)
\end{equation}

The proof of Proposition \ref{prop:main} is concluded.

\section{Proof of Theorem \ref{thm:main}}\label{sec:pfT1}
We will iterate Proposition \ref{prop:main}. Let us start at the trivial solution $(u_0, \pi_0, R_0) \equiv 0$ with $\delta_0 =1$. At the $n$th step we take $\delta_n := {2^{-n}}$ and $\eta_n := \frac{\delta_{n+1}}{2^8 d}$, hence $|R_{n+1} (t)|_{L^1} \le \eta_{n+1}= \frac{1}{2} \frac{\delta_{n+1}}{2^7 d}$. This and $|\mathring A| \le |A| +  \frac{1}{d} |tr A| |\Id| \le  |A| + \frac{1}{d} \sqrt{d} |A| \sqrt{d}$ give $|\cR_{n+1} (t)|_{L^1} \le \frac{\delta_{n+1}}{2^7 d}$, which is the assumption \eqref{eq:Rdelta} of the step $n+1$. Similarly, for any $t \in [0,1]$ at the $n$-th step we get, by \eqref{eq:e_contr}
\[
\frac{3}{8} \delta_n e(t) \le  e(t) - \Big( \int  |u_{n+1}|^2 (t) + 2 \int_0^t \int  \A (D u_{n+1}) D u_{n+1} \Big) \le \frac{5}{8} \delta_{n} e(t)
\]
which is the assumption \eqref{eq:ass_e} of the step $n+1$, since 
\begin{equation*}
\frac{3}{8} \delta_n e(t) = \frac{3}{4} \delta_{n+1}  e(t),   \qquad \frac{5}{8} \delta_{n} e(t) = \frac{5}{4} \delta_{n+1} e(t).
\end{equation*}
Consequently we obtain iteratively, as $\eta_n \le  2^{-n}$,
\begin{equation}\label{eq:istp}
\begin{aligned}
\sup_{t \in [0,1]} |(u_{n+1} - u_n) (t)|_{L_2} & \le M_0 2^{-n/2}, \\
\sup_{t \in [0,1]}  |(u_{n+1} - u_n)  (t)|_{W^{1,q}} & \le 2^{-(n+1)}, \\
\sup_{t \in [0,1]}  |R_{n+1} (t)|_{L_1} & \le 2^{-(n+1)}
\end{aligned}
\end{equation}
Inequalities \eqref{eq:istp} mean that $\{u_n\}_{n=0}^\infty$ is a Cauchy sequence in $C(L^2) \cap C (W^{1,q})$. Denote its limit by $v \in C(L^2) \cap C (W^{1,q})$. Send $n \to \infty$ in the distributional formulation of \eqref{eq:pR}. In particular, in order to pass to the limit in the dissipative term, take a test function $\varphi$ and use \eqref{eq:a-est-1} for $q< \frac{2d}{d+2} <2$ and
the H\"older inequality to obtain 
\begin{equation*}
\left| \int_0^t \int \left[ \A(Du_n) - \A(Dv) \right] \nabla \varphi dxdt \right| \le C
 \left\{ 
\begin{aligned}
 & |\nabla \varphi|_{L^1_t L^q_x}    \sup_{t \in [0,1]} |Du_n - Dv|_{L^q}^{q-1}(t)&\text{ for } \nu_0=0, \\
 & |\nabla \varphi|_{L^1_t L^{q'}_x}    \sup_{t \in [0,1]} |Du_n - Dv|_{L^q}(t) &\text{ for } \nu_0>0.
\end{aligned}
\right.
\end{equation*}
The right-hand sides tend to $0$ as $n \to \infty$ thanks to \eqref{eq:istp}. Consequently we see that $v$ satisfies the distributional formulation of \eqref{eq:pnse}.

For the $2\int_0^t \int \A (D u_n) D u_n$ term of energy we use \eqref{eq:a-est-3} to write
\[
\int_0^t \int |\A(Du_n)Du_n - \A(Dv)Dv | \leq C\int_0^t \int  \big(1 + |Du_n|^{q-1} + |Dv|^{q-1} \big) |Du_n-Dv|.
\]
which via H\"older inequality and \eqref{eq:istp} allows to pass with $n \to \infty$. This and $\lim_{n \to \infty} |u_{n} - v|_{C(L^2)} = 0$ provided by \eqref{eq:istp}   yields \eqref{eq:epr}.

Let us now focus on proving the last part of Theorem \ref{thm:main}, i.e.\ the non-uniqeness statement. Let us take the two energy profiles $e_1, e_2$ and the respective triples $(u^1_n, \pi^1_n, R^1_n)$ and $(u^2_n, \pi^2_n, R^2_n)$ of our convex integration scheme (in what follows, superscripts denote the cases of $e_1, e_2$, respectively). At each iteration step $n \to n+1$ one picks value of $\lambda^i_n$ ($=\lambda$ of Section \ref{ssec:pineq}) that works for $(u^i_n, \pi^i_n, R^i_n)$. Observe that choosing $\bar \lambda_n = \max{(\lambda_n^1, \lambda_n^2)}$ works simultaneously for both triples. Thus, without renaming the triples, let us make the choice $\bar \lambda_n$ for both $(u^i_n, \pi^i_n, R^i_n)$, $i=1,2$. It results in using identical Mikado flows $W^k$ for both iterations. 

Now we want to inductively argue that, thanks to the assumed $e_1(t) = e_2(t)$ for $t \in [0, T]$, it holds $u^1_{n} (t) = u^2_n (t)$ for every $n$ and $t \in [0, T-\frac{1}{2^7 d}]$. Let us assume thence that $u^1_{n} (t)= u^2_n(t)$ and $\cR_n^1 (t)= \cR_n^2 (t)$ for times $t \in [0, T - \sum_{i=0}^n \frac{2^{-i}}{2^7 d}]$ (This holds for $n=0$, since we begin with the zero triple). Formula \eqref{eq:def-ak}, with \eqref{eq:motiv_rho} and \eqref{eq:gnot} shows that $a^i_{k,n+1} (t)$ (i.e.\ every $a^i_k(t)$, $k \in K$  at the step $n \to n+1$) depends on $e(t)$, $\cR_n^{i,\epsilon} (t)$ and $u^i_{n|[0,t]}$, with $\epsilon \le \frac{2^{-(n+1)}}{2^7 d}$ being the mollification parameter, cf.\ \eqref{eq:def-epsilon} with the choice $\delta_{n+1} := {2^{-(n+1)}}$, and the $u^i_n$-dependence being nonlocal due to the dissipative term in \eqref{eq:gnot}. So by our inductive assumption we see that $a^1_{k,n+1} (t) = a^2_{k,n+1}(t)$ for times $t \in [0, T - \sum_{i=0}^{n+1} \frac{2^{-i}}{2^7 d}]$. Consequently, via the definition \eqref{eq:cvp}, the principal perturbations $\vp^i (t)$, $i=1,2$ at the step $n \to n+1$ are identical for times $t \in [0, T - \sum_{i=0}^{n+1} \frac{2^{-i}}{2^7 d}]$. Therefore $u^1_{n+1} (t)= u^2_{n+1} (t)$ and $\cR_{n+1}^1 (t)= \cR_{n+1}^2 (t)$ for $t \in [0, T - \sum_{i=0}^{n+1} \frac{2^{-i}}{2^7 d}]$, since the correctors and the new errors are defined pointwisely in time.

Under the assumption that  $e_1, e_2$ are identical on $[0, T]$, we produced iteratively $u^1_{n} (t), u^2_n (t)$  that agree for $t \in [0, T - \frac{1}{2^7 d}]$ thus also their limits satisfy $v^1 (t)\equiv v^2 (t)$ for $t \in [0, T - \frac{1}{2^7 d}]$.

Replacing $T - \frac{1}{2^7 d}$ with any fixed number $T_1$ strictly smaller than $T$ requires only mollifying at the scales below $T-T_1$ instead of $\frac{1}{2^7 d}$.

\section{Sketch of the proof of Theorem \ref{thm:two}}\label{sec:pfT2}
Let us indicate changes needed in proofs of Proposition \ref{prop:main} and Theorem \ref{thm:main} to reach Theorem \ref{thm:two}. Now, we extend the allowed range of growths of $\A$ to $q \in (1, \frac{3d+2}{d+2})$ at the cost of abandoning the control over the dissipative  term $2 \int_0^t \int  \A (D v) D v$ of the energy. Recall that $r \in (\max\{1, q-1\}, \frac{2d}{d+2})$ is an additional exponent, fixed in the assumptions of Theorem \ref{thm:two}. The main observation is that 
\[
q-1 \le r <\frac{2d}{d+2}
\]
is subcritical in the sense of choices made in section \ref{ssec:mag}.

Let us first consider modifications in proof Proposition \ref{prop:main}. Replacing in sections \ref{ssec:mag}, \ref{ssec:nr} $q$ of Proposition \ref{prop:main} with $r$ implies the following analogue of   \eqref{eq:l-tildel-neg}:
\begin{equation*}
L(r) \leq C\lambda^{-\zeta}, \quad L(2) \leq C\lambda^{-\zeta}, \quad \lambda_2 L({r}) \leq C\lambda^{-\zeta}.
\end{equation*}
for some positive $\zeta>0$. Consequently, \eqref{eq:finwp} holds now also with $r$ in place of $q$. Next, since $q$ now may exceed $2$, to control $|R_\A(t)|_1$ we use the entire \eqref{eq:RA} to write
\[
|R_\A(t)|_1 \le \left\{ 
\begin{aligned}
 & |R_{\A}(t)|_{r}  \le C  \Big(  {\big( \lambda_2 \mu^{\frac{d}{2} - \frac{d}{r}} \big)^{q-1}}  +  (\lambda_2 L(r))^{q-1} \Big) &\text{ for } \nu_0=0, \;q\le2,\\
 & |R_{\A}(t)|_{r}  \le C \big(   {\lambda_2 \mu^{\frac{d}{2} - \frac{d}{r}}} +  \lambda_2 L(r) \big) &\text{ for } \nu_0>0, \;q\le2,\\
 & |R_{\A}(t)|_{{1}}  \le   C \Big( { \lambda_2 \mu^{\frac{d}{2} - \frac{d}{q-1}}} +   { \big( \lambda_2 \mu^{\frac{d}{2} - \frac{d}{q-1}} \big)^{q-1}} \\
 & \qquad \qquad \qquad \qquad  +  \lambda_2 L(q-1) +  (\lambda_2 L(q-1))^{q-1} \Big)&\text{ for } q > 2.
\end{aligned}
\right.
\]
In any of these cases, right-hand sides are controlled by {powers of $\lambda_2 \mu^{\frac{d}{2}-\frac{d}{r}}$ and $\lambda_2 L(r)$}, therefore we can reach \eqref{eq:p1fr}. Finally, since we abandon the control over the dissipative term in the energy inequality, \eqref{eq:gnot} and $\delta E$ of \eqref{eq:Edef} (let us rename it to $\delta \tilde E$) 
loose their dissipative terms. The consequence of the latter is that \eqref{eq:E} simplifies to 
\begin{equation}
\label{eq:E2}
\delta \tilde E(t)  \le \frac{\delta}{16} e(t) + C_N \big({\lambda^{-1}_1} + {\mu^{-\frac{d}{2}}} + L(2) + L(2)^2 \big).
\end{equation}
 Inequality \eqref{eq:E2} allows to prove \eqref{eq:e-delta-8} for $\delta \tilde E$ as in Section \ref{sec:proppf}. 
 
 The above modifications allow to prove Theorem \ref{thm:two} along Section \ref{sec:pfT1} with the difference that now $\{u_n\}_{n=0}^\infty$ forms a Cauchy sequence in $C(L^2) \cap C (W^{1,r})$ and, in the case $q >2$, we pass to the limit in the dissipative term via \eqref{eq:a-est-1} for $q \ge 2$ and the H\"older inequality that give
 \[
 \Big| \int_0^t \int [ \A(Du_n) - \A(Dv) ] \nabla \varphi  \Big| 
 \le C |Du_n-Dv|_{L^{q-1}}  \big(1 + |Du_n|^{q-2}_{L^{q-1}} + |Dv|^{q-2}_{L^{q-1}} \big).
 \]
%
%
\section{Sketch of the proof of Theorem \ref{thm:exist}} 
\label{sec:pfT3}

Let us first introduce the following modification of Definition \ref{def:nnr} 
\begin{defi}
\label{def:cnnr}
Fix $a \in L^2 (\T^d)$. A solution to {\em{the Non-Newtonian-Reynolds Cauchy problem}} is a triple $(u,\pi,R)$ where 
\begin{equation*}
u \in L^\infty(L^2) \cap L^q (W^{1,q}), \quad \pi \in \mathcal{D}, \quad R \in L^1
\end{equation*}
with spatial null-mean $u$, solving the Cauchy problem
\begin{equation}\label{eq:cpR}
\begin{aligned}
\pat u + \divv (u\otimes u)  - \divv \A (\Du) +\nabla \pi&= -\divv \cR, \\
\divv u &= 0, \\
u(0) &= a,
\end{aligned}
\end{equation}
in the sense of distributions, where the data are attained in the weak $L^2$-sense.
\end{defi}

The drop of regularity between objects of Definition \ref{def:cnnr} and objects of Definition \ref{def:nnr} stems from a different starting point for our iterations. To prove Theorems \ref{thm:main}, \ref{thm:two}, we started the iteration at the smooth triple $(u_0, \pi_0, R_0)$ = $(0,0,0)$ and added smooth perturbations in each iteration. To prove Theorem \ref{thm:exist} we will start iterations with $(v_a, \tilde \pi_a, - v_a \otimes v_a)$, where $v_a,  \tilde \pi_a$ solves the Cauchy problem of a non-Newtonian-Stokes system:
\begin{equation}\label{eq:pse}
\begin{aligned}
\pat v_a   - \divv \A( D  v_a)+\nabla \tilde\pi_a &= 0 \\
\divv v_a &= 0 &  \\
v_a (0) &= a
\end{aligned}
\end{equation}

 The smoothness of such $v_a, \tilde \pi_s$ is in general false, so even though at each step we again add smooth perturbations, the regularity at each step cannot be better than that of $v_a, \tilde \pi_a$ solving \eqref{eq:pse}.

The starting point of our iterations is given by
\begin{prop}[Leray-Hopf solutions for non-Newtonian Stokes]\label{prop:nns}
Fix $a \in L^2 (\T^d)$. There is 
\[
v_a \in L^\infty( L^2) \cap L^q (W^{1,q}), \quad \tilde \pi_a \in \mathcal{D}
\]
solving the Cauchy problem \eqref{eq:pse} in the sense of distributions, so that $\lim_{t \to 0} |v_a (t) - a|_2 = 0$ and 
\[
 \int_\Td  |v_a|^2 (t)  + 2 \int_0^t \int_\Td   \A (D v_a) D v_a \le  \int_\Td   |a|^2.
\]
\end{prop}
The proof uses monotonicity of $\A$ and for strong attainment of the initial datum, the energy inequality.

The main ingredient of proof of Theorem \ref{thm:exist} is a version of Proposition \ref{prop:main} tailored to deal with the Cauchy problem.

Roughly speaking, given a solution to the Non-Newtonian-Reynolds system \eqref{eq:cpR}, which assume the given initial datum,  we construct  another solution to \eqref{eq:cpR} \emph{with the same initial datum}, and with a well-controlled Reynolds stress. The price we pay to keep the initial datum intact is growth of energy. More precisely, energy of the ultimately produced solution to the Cauchy problem with datum $a$ for \eqref{eq:pnse} is  \emph{much above} an energy of a non-Newtonian Stokes emanating from the same $a$. Hence we cannot reach energy inequality, even for merely the kinetic energy in the range $q< \frac{2d}{d+2}$. This is why we do not distinguish the subcompact range $q< \frac{2d}{d+2}$ in Theorem  \ref{thm:exist}. 

We are ready to state

\begin{prop}
\label{prop:main3}
Let $\nu_0, \nu_1 \geq 0$ and $q< \frac{3d+2}{d+2}$, $r \in (\max\{1, q-1\}, \frac{2d}{d+2})$ be fixed. Fix an arbitrary nonzero initial datum $a \in L^2(\Td)$, $\divv a = 0$. 
There exist a constant $M$ such the following holds. 

Let $(u_0, \pi_0, R_0)$ be a solution to the Non-Newtonian-Reynolds Cauchy problem with datum $a$. Let us choose any $\delta, \eta, \sigma \in (0,1]$ and $\gamma >0$.
Assume that
\begin{equation}\label{eq:Rdelta3}
|\cR_0 (t)|_{L^1} \le \frac{\delta}{2^7 d} \quad \text{for all } t \in [2\sigma, 1].
\end{equation}
Then, there is a solution $(u_1, \pi_1, R_1)$ to Non-Newtonian-Reynolds  Cauchy problem with same datum $a$ such that
\begin{subequations}\label{eq:mp_all3}
\begin{equation}\label{eq:mp_up3}
|(u_1 - u_0)(t)|_{L^2} \leq
\begin{cases}
M \delta^\frac12 + M \gamma^\frac12 &  t \in [4\sigma,1], \\
M (\delta + \sup_{\tau \in (t-\sigma/4, t+\sigma/4)}   |\cR_0(\tau)|_1 + \gamma)^\frac12  & t \in \left[\sigma/2, 4 \sigma\right], \\
0 & t \in [0, \sigma/2],
\end{cases}
\end{equation} 
\begin{equation}\label{eq:mp_uwp3}
|(u_1 - u_0)(t) |_{W^{1,r}} \le \begin{cases}
\eta & \text{for all } t \in [0, 1], \\
0 & t \in [0, \frac{\sigma}{2}].
\end{cases}
\end{equation}
\begin{equation}\label{eq:mp_R3}
|R_1 (t)|_{L^1} \le 
\begin{cases}
\eta, & t \in [\sigma, 1], \\
|R_0(t)|_{L^1} + \eta, & t \in [\sigma/2, \sigma], \\
|R_0(t)|_{L^1}, & t \in [0, \sigma/2].
\end{cases}
\end{equation}
\end{subequations}
and 
\begin{equation}\label{eq:t_ceL}
\left| |u_1|_2^2 -  |u_0|_2^2 -  d \gamma \right| (t) \le \frac{\delta}{2^4} \qquad t \in [4\sigma, 1],
\end{equation}
\end{prop}
\begin{proof}[Sketch of the proof of Proposition \ref{prop:main3}]
Let us indicate the changes we need to make in the proof of Proposition \ref{prop:main}.

The constant $\gamma$ will be used instead of the energy pump $\gamma_0 (t)$ of \eqref{eq:gnot}. This changes \eqref{eq:motiv_rho} and gives 
\[
\varrho(x,t) \leq 2 \epsilon + 2 |\cR_0^\epsilon(x,t)| + \gamma 
\]
and thus alters \eqref{eq:chipaa} to
\begin{equation}\label{eq:chipaa3}
\left| a_k (t)  \right|_2 \le 2 ( \delta +  |\cR_0(t)|_1 + \gamma)^\frac12.
\end{equation}
Define $u_p$, $u_c$ by \eqref{eq:cvp}, \eqref{eq:cvc} respectively (with the new $\varrho$). 
Let us introduce a smooth cutoff 
\[
\chi(t) 
\begin{cases}
=0 &   t \le \sigma/2, \\
\in [0,1] & t \in (\sigma/2, \sigma) \\
=1, & t \ge \sigma.
\end{cases}
\]
and define the perturbations $\tilde u_p$, $\tilde u_c$ as follows
\begin{equation*}
\tilde u_p(t) := \chi(t) u_p(t), \qquad \tilde u_c(t) := \chi^2(t) u_c(t).
\end{equation*}
Due to \eqref{eq:chipaa3}, the new version of \eqref{eq:uplp2} reads
\begin{equation}\label{eq:uplp23}
|\vp (t)|_2 \le  M ( \delta +  |\cR_0(t)|_1 + \gamma)^{1/2}  +  \lambda_1^{-\frac{1}{2}} C.
\end{equation}
Since  \eqref{eq:Rdelta3} holds only on $[2\sigma,1]$, and recalling the fact that $\cR_0^\epsilon$ is the mollification in space \emph{and} time of $\cR_0$, we can follow the lines of Proposition \ref{prop:main} only on $[4\sigma, 1]$. This gives  the first line of \eqref{eq:mp_up3}. Concerning the case $[\sigma/2, 4\sigma]$ of \eqref{eq:mp_up3} let us compute, using 
\eqref{eq:chipaa3}
\[
|\tilde u_p (t)|^2_{L^2} \le |u_p (t)|^2_{L^2}  \le M ( \delta +  |\cR_0(t)|_1 + \gamma)
\]
Our assumption now does not control $\cR^\epsilon_0(t)$ for $t \le 2\sigma$, but we can always write, choosing $\epsilon \ll \sigma$
\[
|\tilde u_p (t)|^2_{L^2} \leq M \Big(\delta + \sup_{\tau \in (t-\sigma/4, t+\sigma/4)}   |\cR_0(\tau)|_1 + \gamma \Big),
\]
which, together with the smallness of $|\tilde u_c (t)|_{L^2}$, cf.\ \eqref{eq:ucwp}, gives the case $[\sigma/2, 4\sigma]$ of \eqref{eq:mp_up3}. On $[0, \sigma/2]$, it holds $u_1 = u_0$, as there $\chi \equiv 0$, thence the respective part  of \eqref{eq:mp_up3}.

The estimate \eqref{eq:mp_uwp3} holds on the whole time interval, because the Mikado flows are small in $W^{1,r}$ for  $r < \frac{2d}{d+2}$ by construction.

The estimate on the new Reynolds stress \eqref{eq:mp_R3} on $[\sigma,1]$ is analogue to the corresponding estimate \eqref{eq:mp_R} of Proposition \ref{prop:main}, as on $[\sigma,1]$ the time cutoff $\chi \equiv 1$. The estimate on $[0, \sigma/2]$ is trivially satisfied, as on this time interval the cutoff $\chi \equiv 0$ (here $u_0 = u_1$, so  $R_0 = R_1$).
On the intermediate time interval $[\sigma/2,\sigma]$, $R_0$ is decomposed as $R_0 = R_0(1 - \chi^2) + R_0 \chi^2$. The term $R_0 \chi^2$ is canceled by $\tilde u_p \otimes \tilde u_p = \chi^2 (u_p \otimes u_p)$, as in the proof of Proposition \ref{prop:main}, thus giving the $\eta$ in the second line of \eqref{eq:mp_R3}, whereas the term $R_0(1 - \chi^2)$ is responsible for $|R_0(t)|_{L^1}$ in  the second line of \eqref{eq:mp_R3}. There is, however, in the new Reynolds stress $R_1$ an additional term coming from the time derivative of the cutoff:
\begin{equation}
\label{eq:Rcutoff}
\divv R_{cutoff}(t) := \chi'(t) u_p(t) + (\chi^2)'(t) u_c(t).
\end{equation}
For the first term in \eqref{eq:Rcutoff} we just use \eqref{eq:uplp<} with $r=1$:
\begin{equation*}
\begin{aligned}
|\divv^{-1} \big( \chi'(t) u_p(t) \big)|_1 
\leq |\chi'(t) u_p(t)|_1 
\leq \frac{C}{\sigma} \mu^{-d/2} 
\leq \frac{C}{\sigma} \lambda^{-a d/2},
\end{aligned}
\end{equation*}
where we used the choice $\mu= \lambda^a$, as in Section \ref{sec:proppf}. For the second term in \eqref{eq:Rcutoff} we have, invoking  \eqref{eq:inv_stnd}
\begin{equation*}
|\divv^{-1} \big( (\chi^2)'(t) u_c(t) \big)|_1 \leq |\divv^{-1} \big( (\chi^2)'(t) u_c(t) \big)|_2 \leq |(\chi^2)'(t) u_c(t)|_2 \leq \frac{C}{\sigma} L(2)
\end{equation*}
Since by \eqref{eq:l-tildel-neg} $L(2) \leq C \lambda^{-\zeta}$, both terms of  \eqref{eq:Rcutoff} are estimated by negative powers of $\lambda$. Thus they can be made as small as we wish by picking $\lambda$ big enough. 

Let us now justify \eqref{eq:t_ceL} along the proof of Proposition \ref{prop:ei}. In \eqref{eq:el1ip} $\gamma_0$ changes to $\gamma$. Now we do not control $|\cR^\epsilon_0(t)|_1$ on the entire time interval $[0,1]$, only on $[4 \sigma,1]$ via \eqref{eq:Rdelta3}. On this interval $\chi =1$ and hence for any $t \in [4 \sigma,1]$
one has the following counterpart of \eqref{eq:el3}
\begin{equation}\label{eq:el33}
\left| |\tilde u_p|_2^2  - d \gamma \right| \le  \frac{\delta}{2^5} + C \Big( \frac{1}{\lambda_1}   + \frac{\lambda_1 \mu}{\lambda_2} \Big).
\end{equation}
We use now $u_1 = u_0 + \tilde u_p + \tilde u_c$ to write for any $t \in [4 \sigma,1]$
\begin{equation}\label{eq:kec}
\left| |u_1|_2^2 -  |u_0|_2^2  - d \gamma \right| \le  \frac{\delta}{2^5} + C \Big( \frac{1}{\lambda_1}   + \frac{\lambda_1 \mu}{\lambda_2} \Big) +\Big| |\tilde  u_c|_2^2 + 2 \int (u_0 \tilde  u_c + u_0 \tilde  u_p + \tilde u_p \tilde  u_c) \Big|.
\end{equation}
The r.h.s. above can be made arbitrarily small in view of \eqref{eq:el33} and arguments analogous to that leading to Proposition \ref{prop:main},
which yields \eqref{eq:t_ceL}. 
\end{proof}

Iterating Proposition \ref{prop:main3}, we can now complete the proof of Theorem \ref{thm:exist}. Let us choose $\sigma_n := 2^{-n}$ along the iteration. We will choose $\delta_n = {2^{-n}}$ and $\eta_n = {2^{-(n+9)}}d^{-1}$ 
as in proofs of Theorems  \ref{thm:main}, \ref{thm:two}. There are two main differences between the current iterations and the iterations leading to Theorems  \ref{thm:main}, \ref{thm:two}. Firstly, we initiate the iterations with the triple $(v_a, \tilde \pi_a, - v_a \otimes v_a)$, where $v_a, \tilde \pi_a$ are given by Proposition \ref{prop:nns}. Secondly, we will choose the now additional free parameter $\gamma$ just to distinguish between different solutions. Namely, let us choose $\gamma_n = d^{-1} \delta_n$ \emph{except for} $\gamma_3$, which we require it to be a large constant, say $K$.
%
%
%

The condition \eqref{eq:Rdelta3} for the initial triple is void (empty interval where it shall hold) and over iterations it is satisfied thanks to the first case of \eqref{eq:mp_R3} and our choices for $\eta_n, \delta_n, \sigma_n$. The third iteration produces $u_3$ out of $u_2$ such that \[
\left| |u_3|_2^2 -  |u_2|_2^2 -  d K \right| (t) \le {2^{-7}} \qquad t \in [1/2, 1].
\]
At this step the energies of the iteratively produced solutions branch: choosing two $K$'s that considerably differ, we will see that the kinetic energies on $t \in [1/2, 1]$ of the finally produced solutions differ considerably. 

From step $n=4$ onwards $\gamma_n = \delta_n$, thus the first line of \eqref{eq:mp_up3} is analogous to \eqref{eq:mp_up}. 
Iterating Proposition \ref{prop:main3} we thus obtain
convergence of the sequence $\{u_n-v_a\}_n$ to some \[v_\infty \in C((0,1]; L^2(\Td)) \cap C([0,1];  W^{1,r}(\Td)).\] Note the open side of an  interval above. Taking into account the regularity class of $v_a$ and $r<q$, we thus have
\[u_n \to v_a + v_\infty :=v \qquad\text{ strongly  in } \quad L^\infty((0,1); L^2(\Td)) \cap L^r((0,1);  W^{1,r}(\Td)),\]
which allows to pass to the limit in the distributional formulation of \eqref{eq:pse}, since by choice $r>\max{(1, q-1)}$.

Concerning the attainment of the initial datum $a$, for any $q_0 <2$ the estimate \eqref{eq:mp_uwp3} yields $|v_\infty|_{q_0} (t) \to 0$ as $t \to 0$. Therefore $v = v_\infty + v_a$ satisfies $|v - a|_{q_0} (t) \to 0$ as $t \to 0$.
From this and the fact that the $L^2$ norm of $v(t)$ is uniformly bounded in time on $[0,1]$, it follows that $v(t) \rightharpoonup a$ weakly in $L^2$ as $t \to 0^+$. 

Let us finally argue for multiplicity of solutions. At the step $n=3$ let us choose two different $K$, $K'$. Let us distinguish the resulting $u_n$'s and their limits $v$ by, respectively, $u_n$ and $u'_n$, and $v$ and $v'$. 
 On $t \in [1/2, 1]$ \eqref{eq:t_ceL} yields for $n \ge 4$
\[
\left| |u_{n}|_2^2 -  |u_{n-1}|_2^2 \right| (t) \le 2^{-(4+n)} + 2^{-n} \qquad \left| |u'_n|_2^2 -  |u'_{n-1}|_2^2 \right| (t) \le 2^{-(4+n)} + 2^{-n},
\]
 whereas
 \[
\left| |u_3|_2^2 -  |u_2|_2^2 -  d K \right| (t) \le {2^{-7}} \qquad\left| |u'_3|_2^2 -  |u'_2|_2^2 -  d K' \right| (t) \le {2^{-7}}
\]
So, since $u_2 = u'_2$
 \[
\left| |u_{n}|_2^2 -  |u_2|_2^2  -  d K \right| (t) \le 1/2 \qquad \left| |u'_n|_2^2  -  |u_2|_2^2  -  d K'  \right| (t) \le 1/2.
\]
 The same inequalities hold for the strong limits $v$, $v'$. Therefore, for $d|K-K'|>1$ they must differ.



\section{Appendix}\label{sec:app}
\subsection{Proof of Lemma \ref{lem:gA}}
Let us first consider a scalar function $f: \RR \to \RR$
\begin{equation*}
f(t) := (\nu_0 + \nu_1 |t|)^{q-2} t.
\end{equation*}
It is Lipschitz for $\nu_1 =0$. In the range $q \in (1,2]$ $f$ is $(q-1)$-H\"older continuous for  $\nu_0 = 0$, $\nu_1 >0$; and locally Lipschitz for $\nu_0 > 0$, $\nu_1 >0$. By the last statement we mean that 
\begin{equation}\label{eq:afa}
|f(t) - f(s)| \le C_q |t-s|  (\nu_0+\nu_1 (|t|+ |s|))^{q-2}.
\end{equation}
It is proven by writing 
\begin{equation}\label{eq:afb}
\begin{aligned}
|f(t) - f(s) | &\le  |q-1| |t-s| \int_0^1 (\nu_0+\nu_1 |s+ x (t-s)|)^{q-2}  dx \\
& \le |q-1| |t-s| (\nu_0+\nu_1 (|t|+ |s|))^{q-1-\delta} \int_0^1 (\nu_0+\nu_1 |s+ x (t-s)|)^{\delta -1}  dx,
\end{aligned}
\end{equation}
with the second inequality given by splitting $q-2 = (q-1-\delta ) + (\delta -1)$ so that $q-1-\delta>0$ and thus the function $t \to t^{q-1-\delta}$ is increasing. For the integral in the second line of  \eqref{eq:afb} we use that for $\gamma \in (-1, 0]$ and $|t| + |s|>0$ it holds
\begin{equation}\label{eq:af}
 \int_0^1 (\nu_0+\nu_1 |s+ x (t-s)|)^{\delta -1}  dx \le C_\gamma (\nu_0+\nu_1 (|t|+ |s|))^{\gamma}.
\end{equation}
Computation for \eqref{eq:af} is contained in proof of Lemma 2.1 in \cite{AceFus89}. Using  \eqref{eq:af} in  \eqref{eq:afb} we obtain  \eqref{eq:afa}.

In the range $q \ge2$ for $f$, the function $t \to t^{q-2}$ is increasing and integrable at $0$, so \eqref{eq:afa} holds for any $\nu_0 \ge 0$, $\nu_1 \ge 0$. Altogether
\begin{equation}\label{eq:fall}
|f(t) - f(s) | \le \left\{ 
\begin{aligned}
 & C_{\nu_1} |t-s|^{q-1} &\text{ for } \nu_0=0, q<2 \\
 & C_{\nu_0} |t-s| &\text{ for } \nu_0>0, q<2 \\
 & C_{q, \nu_0, \nu_1}  |t-s| \left(1 + |t|^{q-2} + |s|^{q-2} \right) &\text{ for } q\ge2
\end{aligned}
\right.
\end{equation}

Replacing scalars with tensors in \eqref{eq:fall} will give \eqref{eq:a-est-1}. Details follow.

Let us consider first the (global H\"older) case $\nu_0=0, q<2$, i.e.\ the first line of \eqref{eq:fall}. We show the respective first line of \eqref{eq:a-est-1} by considering two cases: (i) $Q, P$ lie along a line passing through the origin and (ii) $Q, P$ lie on a sphere centered at the origin. The case (i) is $Q = t Q_0$, $P = s Q_0$ for some $s,t \in \RR$ and $|Q_0| = 1$, thus \eqref{eq:a-est-1} here follows immediately from \eqref{eq:fall}. The case (ii) is
$|Q| = |P| = l$ for some $l > 0$. Here
\begin{equation*}
\A(Q) =Q  (\nu_1 l)^{q-2}, \quad \A(P) = P  (\nu_1 l)^{q-2},\end{equation*}
so
\begin{equation*}
\frac{|\A(Q)  - \A(P)|}{|Q-P|^{q-1}} = \frac{|Q-P|}{(\nu_1 l)^{2-q}	|Q-P|^{q-1}} =  \Big(\frac{|Q-P|}{\nu_1 l} \Big)^{2-q} \leq \Big( \frac{2l}{\nu_1 l} \Big)^{2-q}
\end{equation*}
and the latter is a  $l$-independent constant.
Both cases (i), (ii) yield for every nonzero $Q,P$
\begin{equation*}
\begin{split}
|\A(Q) - \A(P)|
& \leq \underbrace{\Big| \A(Q) - \A \Big( \frac{|Q|}{|P|} P \Big) \Big|}_{\text{arguments lie: on the sphere $\partial B_{|Q|} (0)$}} + \underbrace{\Big| \A \Big( \frac{|Q|}{|P|} P \Big) - \A(P) \Big|}_{\text{on the same line through origin}} \\
& \leq C_{\nu_1} \Big( \Big| Q - \frac{|Q|}{|P|} P \Big|^{q-1} + \Big| \frac{|Q|}{|P|} P - P \Big|^{q-1} \Big)  \leq C_{\nu_1}  |P-Q|^{q-1}.
\end{split}
\end{equation*}
We have just proven the first $\nu_0=0, q<2$ case of \eqref{eq:a-est-1}.

The second (global Lipschitz) case $\nu_0>0, q<2$, i.e.\ the tensorial version of the second line of \eqref{eq:fall}, can be proven analogously. But in fact the computation leading to \eqref{eq:afa} works well when applied immediately to tensor mappings both in the case $\nu_0>0, q<2$ and $q\ge2$, giving the remainder of \eqref{eq:a-est-1}.
%
Estimate \eqref{eq:a-est-3} follows from an argument that gave the case $q \ge 2$ in \eqref{eq:a-est-1}, applied to $\tilde f(t) = (\nu_0 + \nu_1 |t|)^{q-2} t^2$.

\subsection{Proof of Proposition \ref{prop:invdiv} on antidivergence operators}
\subsubsection*{(i. $\divn: C_0^\infty(\T^d; \RR^d) \to C_0^\infty(\T^d; \cS)$)} Let $\Delta^{-1} v$ denote the null-mean solution $u$ to $\Delta u = v$ on $\Td$. Recall that $D f := \frac{\nabla f + \nabla^T f}{2}$ (symmetric gradient), and recall $P_H f := f - \nabla  \Delta^{-1} \divv f$ (Helmholtz projector). Take \[\divn v := D  \Delta^{-1} v + D P_H \Delta^{-1} v,\] which is symmetric, because $D$ is symmetric. (Compare $\divn$ with the inverse divergence   $\frac32 D  \Delta^{-1} v + \frac12 D P_H \Delta^{-1} v - \frac12 \divv  \Delta^{-1} v \Id$ of Definition 4.2 \cite{DLS13}, which is automatically traceless. We use the simpler choice for $\divn$, since trace zero is provided by a pressure shift.) Since $2 \divv D = \Delta + \nabla \divv$, we have for $\Delta u = v$
\[
\divv (D u + D P_H u) = \frac12 \left(\Delta u + \nabla \divv u + \Delta  (u - \nabla  \Delta^{-1} \divv u) +  \nabla \divv  (u - \nabla  \Delta^{-1} \divv u) \right) = \Delta u = v.
\]
Estimates \eqref{eq:inv_stnd}, \eqref{eq:inv_stnd2} follow from arguments analogous to that of \cite{ModSze18}, proof of Lemma 2.2, so we only sketch them. The estimate
\eqref{eq:inv_stnd} for $p>1$ follows from Calder\'on-Zygmund theory, suboptimally, because $\divn$ is ${-1}$-homogenous. This suboptimality yields the borderline cases. In particular $p=\infty$ holds, since $\nabla \divn$ is $0$-homogenous thus, via Sobolev embedding and Calder\'on-Zygmund for  $\nabla \divn$ one has $|\divn f|_\infty \le C |\nabla \divn f|_{d+1} \le  C |f|_{d+1} \le C |f|_\infty$. The other borderline case $p=1$ follows from the fact that the operator dual to $\divn$ is ${-1}$-homogenous and from duality argument.
The claim \eqref{eq:inv_stnd2} uses  \eqref{eq:inv_stnd}, $-1$-homogeneity of $\divn$ that yields
\begin{equation}\label{eq:osex}
\divn  (u_\lambda) = \lambda^{-1} (\divn u)_\lambda,
\end{equation}
and that we are on a torus, so the $L^p$ norms remain unchanged under oscillations. 

\subsubsection*{(ii. $\idivn_N\!: C^\infty(\T^d; \RR) \times C^\infty_0(\T^d; \RR^d)   \to C^\infty_0(\T^d; \cS)$)} (Compare \cite{DLS13}, Proposition 5.2 and Corollary 5.3.) We construct the two-argument improved symmetric antidivergence iteratively upon $\divn$.
Let us commence with
\begin{equation}\label{eq:R0}
\mathcal{R}_0(f,u) := \divn \Big(f u - \dashint fu  \Big).
\end{equation}
Our aim is an antidivergence that extracts oscillations of $u_\lambda$ out of the product $fu_\lambda$. Therefore \eqref{eq:osex} suggests to apply $\divn$ to $u_\lambda$, and correct the remainder. Let us thence compute
\begin{equation}\label{eq:idivTtov}
\divv \left( f \divn u \right) = f u + \sum_{k=1}^d \partial_k f ( \divv^{-1} u) e_k 
\end{equation}
and define $\mathcal{R}_1 (f, u) := f \divn u -  \sum_{k=1}^d  \idivn_{0} \left(\partial_k f, (\divn u) e_k \right).$
The choice \eqref{eq:R0} of $\mathcal{R}_0$ and \eqref{eq:idivTtov} yield
\begin{equation}\label{eq:idivTtov2}
\divv \mathcal{R}_1 (f, u) = f u - \dashint fu.
\end{equation}
Define inductively
$
\mathcal{R}_N(f, u) := f  \divn u - \sum_{k=1}^d \mathcal R_{N-1} \left(\partial_k f, (\divn u) e_k \right),
$ per analogiam with the construction $\mathcal{R}_0 \to \mathcal{R}_1$.
Using induction over $N$, one proves now that $\idivn_N$ is bilinear, symmetric, $\divv \mathcal{R}_N (f, u) = f u - \dashint_\Td fu$, satisfies Leibniz rule:
\begin{equation*}
\partial_k \mathcal{R}_N (f,u) = \mathcal{R}_N(f, \partial_k u) + \mathcal{R}_N( \partial_k f, u)
\end{equation*}
and estimates \eqref{eq:inv_imprA}, similarly to proof of Lemma 3.5 of \cite{ModSat20}. For instance, to show $\divv \mathcal{R}_N (f, u) = fu - \dashint_\Td fu$ we have \eqref{eq:idivTtov2} as the initial step. Assuming  $\divv \mathcal{R}_{N-1} (f, u) = fu - \dashint_\Td fu$, one has
\[
\divv \mathcal{R}_N (f, u) =  \divv( f \divn u) - \sum_{k=1}^d \Big( \partial_k f (\divn u) e_k - \dashint \partial_k f \left( (\divn u) e_k \right)  \Big)
= f u  - \dashint fu
\]
with the second equality due to \eqref{eq:idivTtov}. 

The estimate \eqref{eq:inv_imprA} holds for $N=1$
since \eqref{eq:inv_stnd} and Jensen imply $| \mathcal{R}_0 (f, u)  |_p \le C |f|_r |u|_s$ and thus
\[
\begin{aligned}
| \mathcal{R}_1 (f, u_\lambda) |_{p} &\le C |f|_r |\divn u_\lambda|_s +  \sum_{k=1}^d  |\idivn_{0} \left(\partial_k f, (\divn u_\lambda) e_k \right)|_p \\
& \le  C|f|_r |\divn  u_\lambda|_s +  C |\nabla f|_r | (\divn u_\lambda) |_s \le C  |u|_s \Big(\frac{1}{\lambda} |f|_r  + \frac{1}{\lambda} |\nabla f|_r \Big) ,
\end{aligned}
\]
with the last inequality due to  \eqref{eq:inv_stnd2}. For the inductive step $N-1 \to N$, let us compute, using \eqref{eq:osex}
\begin{equation*}
\begin{aligned}
&|\mathcal{R}_N(f, u_\lambda )|_{p}  \leq 
 C |f|_r |\divn u_\lambda|_s +\frac{1}{\lambda}\sum_{k=1}^d |  \mathcal{R}_{N-1} \left( \partial_k f, ((\divn u) (\lambda \cdot)) e_k \right) |_p \\
& \leq C |u|_s \frac{1}{\lambda} |f|_r  + C_{d,p,s,r,N-1} \sum_{k=1}^d |\divn u|_s  \frac{1}{\lambda} \Big( \frac{1}{\lambda} |\partial_k f|_r + \frac{1}{\lambda^{N-1}} |\nabla^{N-1} \partial_k f|_r \Big),
\end{aligned}
\end{equation*}
with the second inequality valid via the inductive assumption, i.e.\ \eqref{eq:inv_imprA} for $N-1$.
The estimate \eqref{eq:inv_stnd} and interpolation yield the inductive thesis.

\subsubsection*{(iii. $\tilde \idivn_N: C^\infty(\T^d; \RR^d) \times C^\infty_0(\T^d; \RR^{d \times d})   \to C^\infty_0(\T^d; \cS)$)} Take 
\[
\tilde \idivn_N ( v, T) := \sum_{k=1}^d \idivn_N  (v_k, T e_k ).
\]
Then $\divv \tilde \idivn_N (v, T) = \sum_{k=1}^d v_k T e_k - \dashint_\Td v_k T e_k = Tv -  \dashint_\Td Tv$ and since it is a linear combination of $\idivn_N$'s, it retains all its properties.

\subsubsection*{(iv. $ \idivn^2_N: C^\infty(\T^d; \RR) \times C^\infty_0(\T^d; \RR)   \to C^\infty_0(\T^d; \cS)$)} We redo the reasoning of (i) - (iii), starting from the following `standard double antidivergence' $ \divnd: C_0^\infty(\T^d; \RR) \to C_0^\infty(\T^d; \cS)$
\[
 \divnd v := \nabla^2 \Delta^{-2} v,
\]
where $\nabla^2$ is the (symmetric) tensor of second derivatives, and thus $\divv \divv \divv^{-2} f = f.$ Analogously as for $\divn$ we have
$| \divnd v|_{W^{1,p}} \leq |v|_p  C_p.$
Since it is $-2$-homogenous, it holds 
$\divnd (v_\lambda) = \lambda^{-2} (\divnd v)_\lambda,$
thus
\begin{equation}\label{eq:inv_stnd2d}
| \nabla^i \divnd (v_\lambda)|_{L^{p}} \le \lambda^{i-2} |\nabla^i v|_{L^{p}}  C_{k,p}.
\end{equation} Upon $ \divnd$ we build now $ \idivn^2_N$, starting with
\begin{equation}\label{eq:idivTts0}
\mathcal{R}^2_0(a,b) := \divnd \Big(ab - \dashint ab \Big)
\end{equation}
and the computation 
\begin{equation}\label{eq:idivTtovd}
\divv \divv \left( a  \,\divnd b\right) =  a b + \sum_{k,l=1}^d 2  \partial_{l} (\divnd b)^{kl}  \partial_k a +(\divnd b)^{kl} \partial^2_{kl} a,
\end{equation}
where we used symmetry of $\divnd$.
Hence let us define by recursion
\[
\mathcal{R}^2_N (a, b) := a  \,\divnd b -  \sum_{k,l=1}^d \big[ 2 \mathcal{R}^2_{N-1} \left(\partial_k a,  \partial_{l} (\divnd b)^{kl}  \right) +  \mathcal{R}^2_{N-1} \left(\partial^2_{kl} a, (\divnd b)^{kl}  \right) \big].
\]
Let us inductively prove that 
\begin{equation}\label{eq:idivTtsn}
\divv \divv  \mathcal{R}^2_N (a, b) = ab - \dashint ab.
\end{equation}
The initial statement for $N=0$ is \eqref{eq:idivTts0}. Assuming \eqref{eq:idivTtsn} for $N-1$, we use \eqref{eq:idivTtovd} to compute
\[
\divv \divv \mathcal{R}^2_N (a, b) 
= a b  +  \dashint  \sum_{k,l=1}^d \left( 2\partial_k a \, \partial_l (\divnd b)^{kl} + \partial^2_{kl} a (\divnd b)^{kl} \right).
\]
We see again via \eqref{eq:idivTtovd} that the mean value above equals $- \dashint ab$.

Via induction over $N$, one proves that $\idivn^2_N$ is bilinear, symmetric, and satisfies Leibniz rule. Concerning the estimate \eqref{eq:inv_impr2div}, let us first prove it for $j=0$. The proof is by induction on $N$. For  $N=1$, the estimate is true,  
since $| \mathcal{R}^2_0 (a, b)  |_p \le C |a|_r |b|_s$ and thus 
\[
\begin{aligned}
| \mathcal{R}^2_1 (a, b_\lambda) |_{p} 
& \le  C|a|_r |\divv^{-2}  b_\lambda|_s +  C |\nabla a|_r | (\nabla \divv^{-2} b_\lambda) |_s + C |\nabla^2 a|_r | \divv^{-2} b_\lambda |_s \\
&\le C |b|_s \Big(    \frac{1}{\lambda^2} |a|_r   + \frac{1}{\lambda} |\nabla a|_r +  \frac{1}{\lambda^2} |\nabla^2 a|_r  \Big)
\end{aligned}
\]
with the last inequality due to  \eqref{eq:inv_stnd2d}. For the inductive step $N-1 \to N$, let us compute using $-2$-homogeneity of $\divnd$, $-1$-homogeneity of its derivatives, bililnearity of $\mathcal{R}^2$
\begin{equation*}
|\mathcal{R}^2_N(a, b_\lambda )|_{p}  
  \leq C |b|_s  \frac{1}{\lambda^2} |a|_r +\sum_{k,l=1}^d 2 \frac{1}{\lambda}|  \mathcal{R}^2_{N-1} \left( \partial_k a,  \partial_l (\divnd b)^{kl}_\lambda  \right) |_p + \frac{1}{\lambda^2} |  \mathcal{R}^2_{N-1} \left( \partial^2_{kl} a, (\divnd b )^{kl}_\lambda  \right)|_p,
\end{equation*}
using to the r.h.s.\ above the $j=0$ inductive assumption \eqref{eq:inv_impr2div} for $N-1$ and interpolation yields  \eqref{eq:inv_impr2div} for $j=0$. Finally, estimate \eqref{eq:inv_impr2div} for any $j \in \N$ is proven  by induction on $j$. We already know that \eqref{eq:inv_impr2div} holds for $j=0$. Assuming it is valid for $j$, one has by Leibniz rule and linearity
\[
\nabla^{j+1} \idivn^2_N (a, b_\lambda)  = \nabla^{j} \idivn^2_N (\nabla a, b_\lambda) +\lambda \nabla^{j} \idivn^2_N (a, (\nabla b)_\lambda),
 \]
which via the inductive assumption yields \eqref{eq:inv_impr2div} for $j+1$.

\subsection{Proof of Lemma \ref{l:disj-supp-dmin1}}
The proof is by induction on $|K|$. For $|K| = 1$ the proof is trivial. Let us now assume $|K| \geq 2$. Let us write $K = K' \cup \{k\}$, where $|K'| = |K| -1$. By inductive assumption, $\{\zeta_{k'}\}_{k' \in K'}$ and $\rho'>0$ are already defined, so that the periodization of the cylinders with radius $\rho'$ and axis $\{\zeta_{k'}+sk'\}_{s \in \RR}$, for $k' \in K'$, are pairwise disjoints. It is then enough to find $\zeta_k \in \Rd$ and $\rho \in (0, \rho')$ such that \eqref{eq:disj-supp-d1} holds for all $k' \in K'$. Notice that \eqref{eq:disj-supp-d1} is equivalent to 
\begin{equation*}
B_\rho(\zeta_k) \cap \Big( B_\rho(0) + \big\{ \zeta_{k'} + sk + sk' \}_{s,s' \in \RR } + \Zd \Big) =  \emptyset.  
\end{equation*}
Since $d \ge 3$, the  countable union of planes $\{ \zeta_{k'} + sk + s'k'\}_{s,s' \in \RR} + \Zd$  has zero measure and, since $k,k' \in \Zd$, it is closed in $\Rd$. Therefore, for any $\rho \in (0, \rho')$, also
\begin{equation*}
\bigcup_{k' \in K'} \overline{B_\rho(0)} + \big\{ \zeta_{k'} + sk + s'k'\}_{s,s' \in \RR } + \Zd 
\end{equation*}
is closed in $\Rd$ and, if $\rho$ small enough, it is strictly contained in $\Rd$. We can thus find $\zeta_k$ and $\rho \in (0, \rho')$ such that
\begin{equation*}
B_\rho(\zeta_k) \subseteq \Rd \setminus \Big(\bigcup_{k' \in K'} \overline{B_\rho(0)} + \big\{ \zeta_{k'} + sk + s'k' \, : \, s,s' \in \RR \big\} + \Zd \Big),
\end{equation*}
with the superset being open, thus concluding the proof of the lemma.
\bibliographystyle{abbrv}
\bibliography{biblio_nNnU}

\begin{thebibliography}{10}

\bibitem{AceFus89}
E.~Acerbi and N.~Fusco.
\newblock Regularity for minimizers of non-quadratic functionals: The case $1
  <p < 2$.
\newblock {\em J. Math. Anal. Appl.}, 140(1):115--135, 1989.

\bibitem{BeeBucVic20}
R.~Beekie, T.~Buckmaster, and V.~Vicol.
\newblock Weak solutions of ideal mhd which do not conserve magnetic helicity.
\newblock {\em Ann. PDE}, 1(1), 2020.

\bibitem{BirArmHas87}
Bird, Armstrong, and Hassager.
\newblock {\em Dynamics of Polymer Liquids}.
\newblock John Wiley \& Sons, 1987.

\bibitem{BMR19}
J.~Blechta, J.~Malek, and K.~R. Rajagopal.
\newblock On the classification of incompressible fluids and a mathematical
  analysis of the equations that govern their motion.
\newblock {\em SIAM J. Math. Anal.}, 52(2):1232--1289, 2020.

\bibitem{BCDL}
E.~Bru\'e, M.~Colombo, and C.~De~Lellis.
\newblock Positive solutions of transport equations and classical nonuniqueness
  of characteristic curves.
\newblock {\em arxiv}, arXiv:2003.00539.

\bibitem{BCV19}
T.~Buckmaster, M.~Colombo, and V.~Vicol.
\newblock Wild solutions of the navier-stokes equations whose singular sets in
  time have hausdorff dimension strictly less than $1$.
\newblock {\em arXiv:1809.00600}.

\bibitem{BDLSV18}
T.~Buckmaster, C.~De~Lellis, L.~Sz{\'{e}}kelyhidi, and V.~Vicol.
\newblock Onsager's conjecture for admissible weak solutions.
\newblock {\em Comm. Pure Appl. Math.}, 72(2):229--274, 2018.

\bibitem{BucShkVic19}
T.~Buckmaster, S.~Shkoller, and V.~Vicol.
\newblock Nonuniqueness of weak solutions to the sqg equation.
\newblock {\em Comm. Pure Appl. Math.}, 72(9):1809--1874, 2019.

\bibitem{BucVic19ems}
T.~Buckmaster and V.~Vicol.
\newblock Convex integration and phenomenologies in turbulence.
\newblock {\em EMS Surv. Math. Sci.}, 6(1):173--263, 2019.

\bibitem{BV19}
T.~Buckmaster and V.~Vicol.
\newblock Nonuniqueness of weak solutions to the navier-stokes equation.
\newblock {\em Ann. Math.}, 189(1):101--144, 2019.

\bibitem{BKP19}
M.~Bul{\'{\i}}{\v{c}}ek, P.~Kaplick{\'{y}}, and D.~Pra{\v{z}}{\'{a}}k.
\newblock Uniqueness and regularity of flows of non-newtonian fluids with
  critical power-law growth.
\newblock {\em Mathematical Models and Methods in Applied Sciences},
  29(06):1207--1225, jun 2019.

\bibitem{CheDai19}
A.~Cheskidov and M.~Dai.
\newblock Kolmogorov's dissipation number and the number of degrees of freedom
  for the 3d navier-stokes equations.
\newblock {\em Proc. Roy. Soc. Edinburgh Sect. A}, 149(2):429--446, 2019.

\bibitem{CheLuo19}
A.~Cheskidov and X.~Luo.
\newblock Anomalous dissipation, anomalous work, and energy balance for smooth
  solutions of the navier-stokes equations.
\newblock {\em arXiv:1910.04204}.

\bibitem{CheLuo20}
A.~Cheskidov and X.~Luo.
\newblock Nonuniqueness of weak solutions for the transport equation at
  critical space regularity.
\newblock {\em arXiv:2004.09538}.

\bibitem{CheShv14sima}
A.~Cheskidov and R.~Shvydkoy.
\newblock Euler equations and turbulence: analytical approach to intermittency.
\newblock {\em SIAM J. Math. Anal.}, 46(1):353--374, 2014.

\bibitem{ColDLDR18}
M.~Colombo, C.~D. Lellis, and L.~D. Rosa.
\newblock Ill-posedness of leray solutions for the hypodissipative
  navier-stokes equations.
\newblock {\em Comm. Math. Phys.}, 362(2):659--688, 2018.

\bibitem{DanSze17}
S.~Daneri and L.~Sz{\'e}kelyhidi.
\newblock Non-uniqueness and h-principle for h{\"o}lder-continuous weak
  solutions of the euler equations.
\newblock {\em Arch. Ration. Mech. Anal.}, 224(2):471--514, 2017.

\bibitem{DLS13}
C.~De~Lellis and L.~Sz{\'e}kelyhidi.
\newblock Dissipative continuous euler flows.
\newblock {\em Invent. math.}, 193(2):377--407, 2013.

\bibitem{deW23}
A.~de~Waele.
\newblock Viscometry and plastometry.
\newblock {\em Journal of the Oil \& Colour Chemists Association}, 1923.

\bibitem{DRW10}
L.~Diening, M.~Ruzicka, and J.~Wolf.
\newblock Existence of weak solutions for unsteady motions of generalized
  newtonian fluids.
\newblock {\em Ann. Sc. Norm. Super. Pisa Cl. Sci.}, 9(1):1--46, 2010.

\bibitem{FMS00}
J.~Frehse, J.~Malek, and M.~Steinhauer.
\newblock On existence result for fluids with shear de-pendent viscosity –
  unsteady flows.
\newblock {\em “Partial Differential Equations”, W. Jager, J. Necas, O.
  John, K. Najzar and J. Stara (eds.)}, pages 121--129, 2000.

\bibitem{GFAbook}
M.~Gurtin, E.~Fried, and L.~Anand.
\newblock {\em The Mechanics and Thermodynamics of Continua}.
\newblock Cambridge University Press, 2010.

\bibitem{Ise16}
P.~Isett.
\newblock A proof of onsager's conjecture.
\newblock {\em Ann. Math.}, 188(3):871, 2018.

\bibitem{LadICM66}
O.~A. Ladyzhenskaya.
\newblock On some problems from the theory of continuous media (russian).
\newblock {\em ICM Proceedings, Moscow}, 1966.

\bibitem{DLS09}
C.~D. Lellis and L.~Sz{\'{e}}kelyhidi.
\newblock The euler equations as a differential inclusion.
\newblock {\em Ann. Math.}, 170(3):1417--1436, 2009.

\bibitem{TitLuo20}
T.~Luo and E.~S. Titi.
\newblock Non-uniqueness of weak solutions to hyperviscous {N}avier-{S}tokes
  equations - on sharpness of {J.-L. L}ions exponent.
\newblock {\em Calc. Var. and PDE}, 59(92), 2020.

\bibitem{MNRR}
Malek, Necas, Rokyta, and Ruzicka.
\newblock {\em Weak and measure-valued solutions to evolutionary PDEs}.
\newblock 1996.

\bibitem{MNR93}
J.~Malek, J.~Necas, and M.~Ruzicka.
\newblock On the non-newtonian incompressible fluids.
\newblock {\em Math. Models Methods Appl. Sci.}, 3:35--63, 1993.

\bibitem{ModSat20}
S.~Modena and G.~Sattig.
\newblock Convex integration solutions to the transport equation with full
  dimensional concentration.
\newblock {\em Ann. Henri Poincar\'e C}.

\bibitem{ModSze18}
S.~Modena and L.~Sz{\'e}kelyhidi.
\newblock Non-uniqueness for the transport equation with sobolev vector fields.
\newblock {\em Annals of PDE}, 4(2):18, 2018.

\bibitem{modena-szekelyhidi18}
S.~Modena and L.~Sz{\'e}kelyhidi.
\newblock Non-renormalized solutions to the continuity equation.
\newblock {\em Calc. Var. PDE}, 58:208, 2019.

\bibitem{Nor29}
F.~Norton.
\newblock {\em The Creep os Steel at High Temperatures}.
\newblock McGraw-Hill, 1929.

\bibitem{Ost29}
W.~Ostwald.
\newblock Uber die rechnerische darstellung des strukturgebietes der
  viskositat.
\newblock {\em Kolloid-Zeitschrift}, 1929.

\bibitem{DR20}
L.~D. Rosa.
\newblock Infinitely many leray-hopf solutions for the fractional navier-stokes
  equations.
\newblock {\em arXiv:1801.10235}.

\bibitem{OssRud14}
N.~Rudolph and T.~A. Osswald.
\newblock {\em Polymer Rheology}.
\newblock Hanser Fachbuchverlag, 2014.

\bibitem{Sar16book}
P.~Saramito.
\newblock {\em Complex fluids}.
\newblock Springer International Publishing, 2016.

\bibitem{Sch78book}
W.~Schowalter.
\newblock {\em Mechanics of Non-Newtonian Fluids}.
\newblock Pergamon Press, 1978.

\bibitem{Tan00}
Tanner.
\newblock {\em Engineering Rheology}.
\newblock Oxford University Press, 2000.

\end{thebibliography}

\end{document}